\documentclass[12pt]{amsart}

\usepackage{graphicx}
\usepackage{amsmath}
\usepackage{amsthm}
\usepackage{amscd}
\usepackage{amssymb}
\usepackage{verbatim}
\usepackage[mathscr]{eucal}
\usepackage[all]{xy}
\usepackage{setspace}
\usepackage{psfrag}
\usepackage{url}
\usepackage{pinlabel}

\usepackage[colorlinks=true,linkcolor=blue,citecolor=blue,urlcolor=blue]{hyperref}

\title[The decategorification of bordered Floer homology]{The decategorification of bordered Heegaard Floer homology}
\author{Ina Petkova}
\address {Department of Mathematics, Rice University\\ Houston, TX 77005}
\email {ina@rice.edu}

\theoremstyle{plain}
\newtheorem{theorem}{Theorem}
\newtheorem{proposition}[theorem]{Proposition}
\newtheorem{lemma}[theorem]{Lemma}

\newtheorem{question}[theorem]{Question}

\newtheorem{defn}[theorem]{Definition}

\makeatletter
\renewenvironment{proof}[1][\proofname]{\par
  \pushQED{\qed}%
  \normalfont \topsep6\p@\@plus6\p@\relax
  \trivlist
  \item[\hskip\labelsep
        \bfseries
    #1\@addpunct{.}]\ignorespaces
}{%
  \popQED\endtrivlist\@endpefalse
}
\makeatother

\def\remark{{\bf {\bigskip}{\noindent}Remark. }}
\def\note{{\bf {\bigskip}{\noindent}Note: }}

\def\ackn{{\bf {\bigskip}{\noindent}Acknowledgments. }}

\def\title{\em}

\def\bar{\overline}

  \usepackage[pdftex,lmargin=1.25in,rmargin=1.25in,tmargin=1in,bmargin=1in]{geometry}

\def\example{{\bf {\bigskip}{\noindent}Example: }}

\def\remark{{\bf {\bigskip}{\noindent}Remark. }}
\def\note{{\bf {\bigskip}{\noindent}Note: }}

\def\ackn{{\bf {\bigskip}{\noindent}Acknowledgments. }}
\def\bar{\overline}

\newcommand{\bdy}{\partial}

\newcommand{\ha}{\langle h_A \rangle }
\newcommand{\hd}{\langle h_D \rangle }

\newcommand{\az}{\mathcal{A}(\zz)}
\newcommand{\cala}{\mathcal{A}}
\newcommand{\ainf}{\mathcal{A}_\infty}
\newcommand{\gr}{\textrm{gr}}
\newcommand{\hs}{(\mathcal H, \mathfrak s)}
\newcommand{\ghs}{\mathfrak S(\mathcal H, \mathfrak s)}
\newcommand{\grg}{Grothendieck group }
\newcommand{\sss}{{\bf s}}
\newcommand{\ttt}{{\bf t}}

\newcommand{\fs}{f_{\bf s}}
\newcommand{\xx}{{\bf x}}
\newcommand{\yy}{{\bf y}}

\DeclareMathOperator{\Int}{Int}

\DeclareMathOperator{\inv}{inv}

\DeclareMathOperator{\id}{id}

\DeclareMathOperator{\im}{im}
\DeclareMathOperator{\rank}{rank}

\newcommand{\HH}{\mathcal{H}}

\newcommand{\zz}{\mathcal Z}

\newcommand{\I}{\mathcal{I}}

\newcommand{\Z}{\mathbb Z}
\newcommand{\R}{\mathbb R}

\newcommand{\Q}{\mathbb Q}
\newcommand{\F}{\mathbb F}

\newcommand{\aaa}{\mathbf{a}}
\newcommand{\bbb}{\mathbf{b}}
\newcommand{\xxx}{\mathbf{x}}
\newcommand{\yyy}{\mathbf{y}}
\newcommand{\zzz}{\mathbf{z}}
\newcommand{\www}{\mathbf{w}}

\newcommand{\balpha}{\boldsymbol\alpha}
\newcommand{\bbeta}{\boldsymbol\beta}
\newcommand{\bgamma}{\boldsymbol\gamma}
\newcommand{\bdelta}{\boldsymbol\delta}

\newcommand{\cf}{\mathit{CF}}
\newcommand{\hf}{\mathit{HF}}
\newcommand{\cfd}{\mathit{CFD}}
\newcommand{\cfa}{\mathit{CFA}}
\newcommand{\cfk}{\mathit{CFK}}
\newcommand{\hfk}{\mathit{HFK}}
\newcommand{\kh}{\mathit{Kh}}
\newcommand{\sfh}{\mathit{SFH}}
\newcommand{\cfhat}{\widehat{\cf}}
\newcommand{\hfhat}{\widehat{\hf}}
\newcommand{\cfdhat}{\widehat{\cfd}}
\newcommand{\cfahat}{\widehat{\cfa}}
\newcommand{\cfkhat}{\widehat{\cfk}}
\newcommand{\hfkhat}{\widehat{\hfk}}
\newcommand{\cfkm}{\cfk^-}

\newcommand{\cfam}{\cfa^-}

\newcommand{\brho}{\boldsymbol\rho}

\begin{document}

\maketitle

\begin{abstract}
Bordered Heegaard Floer homology is an invariant for $3$-manifolds, which associates to a surface $F$ an algebra  $\az$, and to a $3$-manifold $Y$ with boundary, together with an orientation-preserving diffeomorphism $\phi: F\to \bdy Y$, a module over $\az$.
We study the Grothendieck group of modules over $\az$, and define an invariant lying in this group for every bordered 3-manifold $(Y, \bdy Y, \phi)$. We prove that this invariant  recovers the kernel of the inclusion $i_\ast:H_1(\bdy Y; \Z)\to H_1(Y; \Z)$ if $H_1(Y, \bdy Y; \Z)$ is finite, and is $0$ otherwise. We also study the properties of this invariant corresponding to gluing. As one application, we show that the pairing theorem for bordered Floer homology categorifies the classical Alexander polynomial formula for satellites. 
\end{abstract}

\section{Introduction}

Heegaard Floer homology is an approach, motivated by gauge theory, to studying knots, links, and 3- and 4-manifolds, developed by Ozsv\'ath and Szab\'o.  To a closed 3-manifold $Y$ one associates a graded chain complex $\cfhat(Y)$ whose chain homotopy type is a powerful homeomorphism invariant of the manifold. A knot $K$ in a 3-manifold $Y$ induces a filtration on the complex $\cfhat(Y)$, which in turn leads to knot Floer homology---a bigraded homology theory for knots. Amongst the many valuable properties of knot Floer homology, the simplest version, $\hfkhat(Y,K)$, categorifies the Alexander polynomial, detects the genus, and detects fiberedness. 

The ideas of Heegaard Floer homology were recently generalized by Lipshitz, Ozsv\'ath and Thurston to 3-manifolds with boundary \cite{bfh2}. The new theory, bordered Heegaard Floer homology, provides powerful gluing techniques for computing the original Heegaard Floer invariants of closed manifolds and knots. 

We explore the structural aspects of the bordered theory, developing the notion of an Euler characteristic for each of the two types of modules associated to a bordered manifold. The Euler characteristics of related Floer homologies have been shown to be invariants of $3$-manifolds. For example, for a closed manifold $Y$, the Euler characteristic of $\hf^+(Y)$ is the Turaev torsion of $Y$, and for a knot $K$ in $S^3$, the Euler characteristic of $\hfkhat(K)$ is the Alexander polynomial $\Delta_K(t)$. For sutured manifolds, Juh\'asz developed a Floer theory called sutured Floer homology \cite{sfh}, whose Euler characteristic has been shown to be a certain Turaev-type torsion function \cite{dsfh}. 
In this paper, we study the Euler characteristic of bordered Floer homology, its relation to the Euler characteristics of the aforementioned Floer theories, and its behavior under gluing. 

Bordered Floer homology is a TQFT-type generalization of $\hfhat$ to manifolds with boundary. To a parametrized surface one associates a differential graded algebra $\az$, where $\zz$ is a way to represent the surface, and to a manifold with parametrized boundary represented by $\zz$ a left type $D$ structure $\cfdhat$ over $\az$, or a right $\ainf$-module $\cfahat$ over $\az$.  Both structures are invariants of the manifold up to homotopy equivalence, and their derived tensor product is an invariant of the closed manifold obtained by gluing two bordered manifolds along their boundary, and recovers $\cfhat$. Another variant of these structures is associated to knots in bordered 3-manifolds, and recovers $\cfkhat$ after gluing.

We study the  \grg of the surface algebra $\az$, and prove that the image of the above structures in this group is an invariant of the bordered manifold. The difficulty in obtaining an interesting invariant lies in the fact that there is no differential $\Z$-grading on the algebra and modules. Instead, $\az$ is graded by a non-abelian group $G$, and the modules are graded by sets with a $G$-action. An Euler characteristic which carries no grading data loses too much information about the manifold, while one carrying the full data from $G$ would be harder to define, as well as to interpret and relate to its sisters in the closed and sutured worlds. 

To obtain an invariant with integer coefficients, we define a $\Z/2$ differential grading $m$ on the algebra and modules, and show that it agrees with the Maslov grading under gluing.

Suppose $M$ is a right $\ainf$-module over $\az$ with a set of  ``homogeneous" generators $\mathfrak S(M)$, i.e. for each generator $x$ there is a unique indecomposable idempotent $I(x)$ that acts non-trivially (by the identity) on that generator. In Section \ref{k0sec}, we define a correspondence between indecomposable idempotents of $\az$ and generators of $\Lambda^\ast(H_1(F; \Z))$, and a map $h:\mathfrak S(M)\to \Lambda^\ast(H_1(F; \Z))$ sending each generator $x$ to the generator of $\Lambda^\ast(H_1(F; \Z))$ corresponding to $I(x)$. In Section \ref{z2sec} we define a differential grading $m$ on the algebra $\az$ by $\Z/2$. In Section \ref{k0sec}, we prove the following:

\begin{theorem}\label{intro_k0}
Let $\zz$ be a pointed matched  circle with associated surface $F$ of genus $k$. The Grothendieck group of the category of $\Z/2$-graded $\az$-modules is given by
$$K_0(\az) \cong  \Lambda^\ast(H_1(F; \Z))\cong \Z^{2^{2k}}.$$ 
Moreover,  for an $\az$-module $M$ as above, its image in this group is given by
$$[M] = \sum _{x \in\mathfrak S(M)} (-1)^{m(x)}h(x),$$
where $m$ is the grading of $M$ by  $\Z/2$.
\end{theorem}

In other words, the Euler characteristic of an $\az$-module counts generators according to their grading, and the algebra action on them. 

To formulate the behavior of the Euler characteristic under gluing, note that the  tensor product $M\boxtimes N$ is just a chain complex, and so its Euler characteristic is an integer. Thus, gluing should correspond to pairing $[M]$ and $[N]$ in some way and interpreting the result as an integer.  Specifically, for $a, b \in K_0(\az)$, we define a product $a\cdot b\in \Z$ in Section \ref{tensorsec}.

\begin{theorem}\label{intro_tensor}
Let $M$ be a right $\cala_\infty$-module over $\az$ and $N$ a left type $D$ structure  over $\az$. The Euler characteristic of the chain complex $M\boxtimes N$ is  $\chi(M\boxtimes N) = [M]\cdot[N]$. In particular, let $Y_1$ and $Y_2$ be bordered $3$-manifolds which agree along their boundary, with Heegaard diagrams $\mathcal H_1$ and $\mathcal H_2$, so that $\mathcal H = \mathcal H_1\cup_\bdy \mathcal H_2$ represents the closed $3$-manifold $Y = Y_1\cup_\bdy Y_2$. Fix  $\mathfrak s_i \in \mathrm{spin}^c(Y_i)$. Then
 up to an overall sign
$$[\cfahat(\mathcal H_1, \mathfrak s_1)] \cdot [ \cfdhat(\mathcal H_2, \mathfrak s_2)] = \chi\left(\bigoplus_{\substack{\mathfrak s\in \mathrm{spin}^c(Y)   \\ \mathfrak s|{Y_i} = \mathfrak s_i, i = 1,2 }}\cfhat(\HH, \mathfrak s)\right).$$
\end{theorem}

A similar statement can be made when one of the bordered manifolds is endowed with a knot. For that purpose, we define a second (internal) grading which behaves much like the Alexander grading for knots in closed manifolds, and study the Euler characteristic in $\Lambda^\ast(H_1(F; \Z))\otimes \Z[t, t^{-1}]$, where $t$ corresponds to the internal grading. One topological significance of this new invariant is that it recovers the Alexander polynomial.

Recall that given knots $K\hookrightarrow S^3$ and  $C\hookrightarrow S^1\times D^2$, we can glue $S^1\times D^2$ to $S^3\setminus \nu(K)$ by identifying the $0$-framings. The image of $C$ after this gluing is a \emph{satellite knot} denoted $K_C$. Let $k = \#(C\cap D^2)$ be the homology class of $C$ inside $S^1\times D^2$. It is a classical result that $\Delta_{U_C}(t)\cdot \Delta_K(t^k) = \Delta_{K_C}(t)$, where $U\hookrightarrow S^3$ is the unknot. The pairing on bordered Floer homology categorifies this formula.

To state the result, for a bigraded module $M$ we denote by  $[M, k]$  the Euler characteristic of $M$ where $t$ is replaced by $t^k$.  
Also note that the pairing in Section \ref{tensorsec} has a bigraded version
$$\cdot:\Lambda^\ast H_1(F; \Z)\otimes \Z[t, t^{-1}]\times \Lambda^\ast H_1(F; \Z)\otimes \Z[t, t^{-1}]\to \Z[t, t^{-1}].$$

\begin{theorem}\label{intro_alex}
Given  oriented knots $K$ in $S^3$ and  $C$ in a $0$-framed $S^1\times D^2$, let $k = \#(C\cap D^2)$ be the homology class of $C$ inside $S^1\times D^2$. Let $(S^3\setminus K, 0)$ be the $0$-framed complement of $K$, and let $a_1$ and $a_2$ be  the generators of $H_1(T^2; \Z)$ corresponding to the $0$-framing and the $\infty$-framing. Then 
$$[\cfahat(S^1\times D^2, C)]\cdot [\cfdhat(S^3\setminus K, 0),k] = \chi (\cfkhat(K_C)).$$
Moreover, 
$$[\cfahat(S^1\times D^2, C)] =  \chi(\cfkhat(U_C)) a_1  + P_C(t) a_2=  \Delta_{U_C}(t)a_1 +  P_C(t) a_2,$$
and
$$[\cfdhat(S^3\setminus K, 0)] =  \chi(\cfkhat(K)) a_1  =  \Delta_K(t) a_1.$$
In other words, the decategorification of bordered Floer homology in this case is precisely the classical   Alexander polynomial formula for satellites
$$\Delta_{U_C}(t)\cdot \Delta_K(t^k) = \Delta_{K_C}(t).$$
\end{theorem}
We would like to study the polynomial $P_C(t)$ further to see what additional information one might gain about the satellite $C$.

Our last main result discusses the topological data that $[\cfdhat(Y)]$ encodes.

Work of Zarev \cite{bs} implies that the Euler characteristic of bordered Floer homology recovers the Euler characteristic of sutured Floer homology for the same manifold with boundary with properly chosen sutures. The Euler characteristic of sutured Floer homology has some interesting topological properties \cite{dsfh}.  We  hope that since bordered Floer homology also  satisfies a nice gluing formula, we can find new topological information in its Euler characteristic. 

 As a first step in that direction, we make a statement about $[\cfd]$ without an additional Alexander grading. Below, all homology coefficients are in $\Z$. Looking into this was suggested to the author by Peter Ozsv\'ath, as a possible one-boundary-component Heegaard Floer analogue to Donadson's formulas from \cite{sd}.
 
\begin{theorem}\label{cfdker}
Let $(Y, \zz, \phi)$ be a bordered $3$-manifold, with boundary of some genus $k$. 
 If $b_1(Y, \bdy Y) >0$, then $[\cfdhat(Y)] = 0$. If $b_1(Y, \bdy Y) = 0$, then 
  $$\mathrm{span} [\cfdhat(Y)] = |H_1(Y, \bdy Y)| \Lambda^k \ker(i_{\ast}:H_1(F(\zz))\to H_1(Y)).$$
\end{theorem}

Note that $[\cfdhat(Y)]$ is an element of $\Lambda^k(H_1(F(\zz)))$, and $\mathrm{span} [\cfdhat(Y)]$ is the $\Z$ span of this element.

It would be interesting to explore further the additional structure of the Grothendieck group of a surface algebra, for example, the structure arising from the action of the mapping class group, or from other types of surface cobordisms. Similar to  the surgery formulae for Turaev torsion, one should be able to  obtain new formulae when the gluing is along any genus surface.

\ackn I am  deeply thankful to Robert Lipshitz for suggesting this problem and for his close guidance throughout. I am grateful to Peter Ozsv\'ath for  numerous enlightening conversations, and for pointing my attention to what became Theorem \ref{cfdker}. I thank Alexander Ellis, Dylan Thurston and Rumen Zarev for helpful discussions,  Sucharit Sarkar for his comments on an earlier version of this paper, and the referee for many useful suggestions and corrections.

\section{Background in bordered Floer homology}

This section is a brief introduction to bordered Floer homology.

\subsection{The algebra}
In this section, we describe the differential graded algebra $\az$ associated to the parametrized boundary of a $3$-manifold. For further details, see  \cite[Chapter 3]{bfh2}.

\begin{defn}\label{alg}
The \emph{strands algebra} $\cala(n,k)$ is a free $\Z/2$-module generated by partial permutations $a = (S, T, \phi)$, where $S$ and $T$ are $k$-element subsets of the set $[n]:= \{1, \ldots, n\}$ and $\phi:S\to T$ is a non-decreasing bijection. We let $\inv(a) = \inv(\phi)$ be the number of inversions of $\phi$, i.e. the number of pairs $i,j\in S$ with $i<j$ and $\phi(j)<\phi(i)$. Multiplication is given by 
\begin{displaymath}
(S, T, \phi)\cdot(U, V, \psi) = \left\{ \begin{array}{ll}
(S, V, \psi\circ\phi) & \textrm{if $T=U$, $\inv(\phi)+ \inv(\psi) = \inv(\psi\circ\phi)$}\\
0 & \textrm{otherwise.}
\end{array} \right.
\end{displaymath}
See  \cite[Section 3.1.1]{bfh2}.
We can represent a generator $(S, T, \phi)$ by a strands diagram of horizontal and upward-veering strands. See \cite[Section 3.1.2]{bfh2}.
The differential of $(S, T, \phi)$ is the sum of all possible ways to ``resolve" an inversion of $\phi$ so that $\inv$ goes down by exactly $1$. Resolving an inversion $(i,j)$ means switching $\phi(i)$ and $\phi(j)$, which graphically can be seen as smoothing a crossing in the strands diagram.
\end{defn}

The ring of idempotents $\mathcal I(n,k)\subset \cala(n,k)$ is generated by all elements of the form $I(S) := (S, S, \textrm{id}_S)$ where $S$ is a $k$-element subset of $[n]$. 
 
\begin{defn} A \emph{pointed matched circle} is a quadruple $\zz = (Z, {\bf a}, M, z)$ consisting of an oriented circle $Z$, a collection of $4k$ points ${\bf a} = \{a_1, \ldots, a_{4k}\}$ in $Z$, a \emph{matching} of ${\bf a}$, i.e., a $2$-to-$1$ function $M:{\bf a}\to [2k]$, and a basepoint $z\in Z\setminus {\bf a}$. We require that performing oriented surgery along the $2k$ $0$-spheres $M^{-1}(i)$ yields a single circle. 
\end{defn}

A matched circle specifies a handle decomposition of an oriented surface $F(\zz)$ of genus $k$: take a $2$-dimensional $0$-handle with boundary $Z$, $2k$ oriented $1$-handles attached along the pairs of matched points, and a $2$-handle attached to the resulting boundary. 

If we forget the matching on the circle for a moment, we can view $\cala(4k) = \bigoplus_{i}\cala(4k, i)$ as the algebra generated by the Reeb chords in $(Z\setminus z, {\bf a})$: We can view a set ${\boldsymbol{\rho}}$ of Reeb chords, no two of which share initial or final endpoints, as a strands diagram of upward-veering strands. For such a set ${\boldsymbol{\rho}}$, we define the \emph{strands algebra element associated to ${\boldsymbol{\rho}}$} to be the sum of all ways of consistently adding horizontal strands to the diagram for ${\boldsymbol{\rho}}$, and we denote this element by $a_0({\boldsymbol{\rho}})\in \cala(4k)$. The basis over $\Z/2$ from Definition \ref{alg} is in this terminology the non-zero elements of the form $I(S)a_0({\boldsymbol{\rho}})$, where $S\subset \bf a$.

For a subset ${\bf s}$ of $[2k]$, a \emph{section} of ${\bf s}$ is a set  $S\subset M^{-1}({\bf s})$, such that $M$ maps $S$ bijectively to ${\bf s}$. To each ${\bf s}\subset [2k]$ we associate an idempotent in $\cala (4k)$ given by
$$I({\bf s}) = \sum_{S \textrm{ is a section of } {\bf s}} I(S).$$
Let $\mathcal I(\zz)$ be the subalgebra generated by all $I({\bf s})$, and let ${\bf I} = \sum_{\bf s}I({\bf s})$.
\begin{defn}
The \emph{algebra $\az$ associated to a pointed matched circle $\zz$} is the subalgebra of $\cala(4k)$ generated (as an algebra) by $\mathcal I(\zz)$ and by all $a({\boldsymbol{\rho}}) :={\bf I}a_0({\boldsymbol{\rho}})\bf{ I}$. We refer to $a({\boldsymbol{\rho}})$ as the \emph{algebra element associated to ${\boldsymbol{\rho}}$}.
\end{defn}
We will view $\cala(\zz)$ as defined over the ground ring $\mathcal I(\zz)$. Note that $\cala(\zz)$ decomposes as a direct sum of differential graded algebras
$$\cala(\zz) = \bigoplus_{i = -k}^k\cala (\zz, i),$$
where $\cala (\zz, i) = \cala(\zz)\cap \cala(4k, k+i)$. Similarly, set $\mathcal I(\zz, i) = \mathcal I(\zz)\cap \cala(4k, k+i)$.

\subsection{Type $D$ structures, $\cala_\infty$-modules, and tensor products}

We recall the definitions of the algebraic structures used in \cite{bfh2}. For a beautiful, terse description of type $D$ structures and their basic properties, see \cite[Section 7.2]{bs}, and for a more general and detailed description of $\cala_\infty$ structures,  see \cite[Section 2.3]{bfh2}. 

Let $A$ be a unital differential graded algebra with differential $d$ and multiplication $\mu$ over a base ring ${\bf k}$. In this paper, ${\bf k}$ will always be a direct sum of copies of $\F_2 = \Z/2\Z$. When the algebra is $\cala(\mathcal Z)$, the base ring for all modules and tensor products is $\mathcal I(\mathcal Z)$.

A \emph{(right) $\cala_\infty$-module} over $A$ is a graded module $M$ over ${\bf k}$, equipped with maps 
$$m_i: M\otimes A^{\otimes (i-1)}\to M[2-i],$$ 
satisfying the compatibility conditions
\begin{align*}
0&= \sum_{i+j = n+1}m_i(m_j(\xxx, a_1, \cdots, a_{j-1}), \cdots , a_{n-1})\\
&+ \sum_{i=1}^{n-1} m_n(\xxx, a_1,\cdots, a_{i-1}, d(a_i),\cdots, a_{n-1})\\
&+ \sum_{i=1}^{n-2} m_{n-1}(\xxx, a_1,\cdots, a_{i-1}, (\mu(a_i,a_{i+1})),\cdots, a_{n-1})
\end{align*}
and the unitality conditions $m_2(\xxx, 1) = \xxx$ and $m_i(\xxx, a_1,\cdots, a_{i-1})=0$ if $i>2$ and some $a_j =1$. We say that $M$ is \emph{bounded} if  $m_i=0$ for all sufficiently large $i$.

A \emph{(left) type $D$ structure}  over $A$ is a graded module $N$ over the base ring, equipped with a homogeneous map
$$\delta:N\to (A\otimes N)[1]$$
satisfying the compatibility condition
$$(d\otimes \id_N)\circ \delta + (\mu\otimes \id_N)\circ (\id_A\otimes \delta)\circ \delta= 0.$$
We can define maps
$$\delta_k:N\to (A^{\otimes k}\otimes N)[k]$$
inductively by
$$\delta_k = \left\{  \begin{array}{ll}
\id_N & \textrm{ for } k=0\\
(\id_A\otimes \delta_{k-1})\circ \delta & \textrm{ for } k\geq 1 
\end{array}\right.$$

If $(N, \delta)$ is a type $D$ structure, then $A\otimes N$ can be given the structure of a differential module over $A$, with differential
$$\bdy = \mu_1\otimes \id_N + (\mu_2\otimes \id_N)\circ \delta$$
and algebra action
$$a\cdot (b,n) = (\mu_2(a,b),n).$$

 A type $D$ structure is said to be \emph{bounded} if for any $\xxx \in N$, $\delta_i(\xxx) = 0$ for all sufficiently large $i$.

If $M$ is a right $\cala_\infty$-module over $A$ and $N$ is a left type $D$ structure, and at least one of them is bounded, we can define the \emph{box tensor product} $M\boxtimes N$ to be the  vector space $M\otimes N$ with differential 
$$\bdy: M\otimes N \to M\otimes N[1]$$
defined by 
$$\bdy = \sum_{k=1}^{\infty}(m_k\otimes \id_N)\circ(\id_M\otimes \delta_{k-1}).$$
The boundedness condition guarantees that the above sum is finite. In that case $\bdy^2= 0$ and $M\boxtimes N$ is a graded chain complex.

Given two differential graded algebras, four types of bimodules can be defined in a similar way. We omit those definitions and refer the reader to \cite[Section 2.2.4]{bimod}.

\subsection{Bordered three-manifolds, Heegaard diagrams, and their modules}

A \emph{bordered $3$-manifold} is  a triple $(Y, \zz, \phi)$, where $Y$ is a compact, oriented $3$-manifold with connected boundary $\bdy Y$, $\zz$ is a pointed matched circle, and $\phi: F(\zz)\to \bdy Y$ is an orientation-preserving homeomorphism. A bordered $3$-manifold may be represented by a \emph{bordered Heegaard diagram}  $\mathcal H =(\Sigma, \bbeta, \balpha, z)$, where $\Sigma$ is an oriented surface of some genus $g$  with one boundary component, $\bbeta$ is a set of pairwise-disjoint, homologically independent circles in $\Int(\Sigma)$, $\balpha$ is a $(g+k)$-tuple of pairwise-disjoint curves in $\Sigma$, split into $g-k$ circles in $\Int(\Sigma)$, and $2k$ arcs with boundary on $\bdy \Sigma$, so that they are all homologically independent in $H_1(\Sigma, \bdy \Sigma)$, and $z$ is a point on $(\bdy\Sigma)\setminus (\balpha\cap \bdy \Sigma)$.
The boundary $\bdy\mathcal H$ of the Heegaard diagram has the structure of a pointed matched circle, where two points are matched if they belong to the same $\alpha$-arc. 
We can see how a bordered Heegaard diagram $\mathcal H$ specifies a bordered manifold  in the following way. Thicken up the surface  to $\Sigma\times [0,1]$, and attach a three-dimensional two-handle to each circle $\alpha_i\times \{0\}$, and a three-dimensional two-handle to each $\beta_i\times \{1\}$. Call the result $Y$, and let $\phi$ be the natural identification of $F(\bdy \mathcal H)$ with $\bdy Y$. Then $(Y, \bdy \mathcal H, \phi)$ is the bordered 3-manifold for $\mathcal H$.

To a bordered Heegaard diagram $(\mathcal H, z)=(\Sigma, \balpha, \bbeta, z)$, we associate either a left type $D$ structure $\cfdhat (\mathcal H, z)$ over $\mathcal A(-\bdy\mathcal H)$, or a right $\mathcal A_{\infty}$-module $\cfahat (\mathcal H, z)$ over $\mathcal A(\bdy\mathcal H)$. Similarly, we can represent a knot in a bordered $3$-manifold by a doubly-pointed bordered Heegaard diagram $(\mathcal H, z, w) = (\Sigma, \balpha, \bbeta, z, w)$, where $z$ and $w$ are in $\Sigma\setminus(\balpha\cup \bbeta)$, and  $z\in \bdy \mathcal H$. To this diagram we can associate a right $\mathcal A_{\infty}$-module $\cfam(\mathcal H, z, w)$, this time over $\F_2[U]$, where a holomorphic curve passing through $w$ with multiplicity $n$ contributes $U^n$ to the multiplication. Setting $U=0$ gives  $\cfahat(\mathcal H, z, w)$, where we count only holomorphic curves that do not cross $w$.

Now we define the above modules.
A \emph{generator} of a bordered Heegaard diagram $\mathcal H = (\Sigma, \balpha, \bbeta)$ of genus $g$ is a $g$-element subset $\xxx = \{x_1, \ldots, x_g\}$ of $\balpha\cap \bbeta$, such that there is exactly one point of $\xxx$  on each $\beta$-circle, exactly one point on each $\alpha$-circle, and at most one point on each $\alpha$-arc. Let $\mathfrak S(\mathcal H)$ denote the set of generators. Given $\xxx\in \mathfrak S(\mathcal H)$, let $o(\xxx)\subset [2k]$ denote the set of $\alpha$-arcs occupied by $\xxx$, and let $\bar o(\xxx)=[2k]\setminus o(\xxx)$ denote the set of unoccupied arcs.

Fix generators $\xxx$ and $\yyy$, and let $I$ be the interval $[0,1]$. Let $\pi_2(\xxx, \yyy)$, the \emph{homology classes from $\xxx$ to $\yyy$}, denote the elements of 
$$H_2(\Sigma\times I\times I, ((\balpha\times \{1\}\cup \bbeta\times \{0\}\cup(\bdy\Sigma\setminus z)\times I)\times I)\cup (\xxx\times I\times \{0\})\cup (\yyy\times I\times\{1\}))$$
which map to the relative fundamental class of $\xxx\times I\cup \yyy\times I$ under the composition of the boundary homomorphism and collapsing the rest of the boundary.

A homology class $B\in \pi_2(\xxx, \yyy)$ is determined by its \emph{domain}, the projection of $B$ to $H_2(\Sigma, \balpha\cup\bbeta\cup \bdy\Sigma)$. We can interpret the domain of $B$ as a linear combination of the components, or \emph{regions}, of $\Sigma\setminus (\balpha\cup \bbeta)$.

Concatenation at $\yyy\times I$, which corresponds to addition of domains, gives a product $\ast:\pi_2(\xxx, \yyy)\times \pi_2(\yyy, \www)\to \pi_2(\xxx, \www)$. This operation turns $\pi_2(\xxx, \xxx)$ into a group called the group of \emph{periodic domains}, which is naturally isomorphic to $H_2(Y, \bdy Y)$.

Let $X(\mathcal H)$ be the $\F_2$ vector space spanned by $\mathfrak S(\mathcal H)$. Define $I_D(\xx)= \bar o(\xx)$. We define an action  on $X(\mathcal H)$ of $\mathcal I(-\bdy\mathcal H)$ by
$$I(\sss)\cdot \xx = \left\{  \begin{array}{ll}
\xx & \textrm{ if } I(\sss) = I_D(\xx)\\
0 & \textrm{ otherwise. } 
\end{array}\right.$$
Then $\cfdhat(\mathcal H)$ is defined as an $\mathcal A(-\bdy\mathcal H)$-module by 
$$\cfdhat (\mathcal H) = \mathcal A(-\bdy\mathcal H)\otimes_{\mathcal I(-\bdy\mathcal H)}X(\mathcal H).$$
Since there are no explicit computations of $\cfdhat$ in this paper, we omit the definition of  the map $\delta_1$ and refer the reader to \cite[Section 6.1]{bfh2}. 

Define $I_A(\xx)=  o(\xx)$. The module $\cfahat(\mathcal H)$ is generated over $\F_2$ by $X(\mathcal H)$, and the right action of $\mathcal I(\bdy \mathcal H)$ on $\cfahat(\mathcal H)$ is defined by
$$\xx\cdot I(\sss) = \left\{  \begin{array}{ll}
\xx & \textrm{ if } I(\sss) = I_A(\xx)\\
0 & \textrm{ otherwise. } 
\end{array}\right.$$
The $\ainf$ multiplication maps count certain holomorphic representatives of the homology classes defined in this section \cite[Definition 7.3]{bfh2}.

\subsection{Gradings}\label{m}

It turns out there is no $\Z$-grading on  $\mathcal A(\zz)$. Instead, the algebra is graded by a nonabelian group $G$, and the  left or right modules over it are graded by left or right cosets of  a subgroup of $G$. In the topological case, domains of a Heegaard diagram can be also graded by $G$, and  the subgroup above is the image of periodic domains in $G$. 
We briefly recall the relevant definitions and results. For more detail, see \cite[Sections 2.5, 3.3, and 10]{bfh2}.

\subsubsection{Gradings by nonabelian groups}
Before we specialize to bordered Floer homology, we recall  gradings by nonabelian groups in general. 
\begin{defn} \label{def:group_gr}
Let $G$ be a group and $\lambda$ be an element  in the center of $G$. A \emph{differential algebra graded by $(G, \lambda)$} is a differential algebra $A$ with a grading $\gr$ by $G$ (as a set), i.e.  a decomposition $A = \bigoplus_{g\in G}A_g$. We say that an element $a\in A_g$ is \emph{homogeneous of degree $g$}, and write $\gr(a) = g$. We also require that for homogeneous $a$ and $b$ the grading is
\begin{itemize} 
\item compatible with the product, i.e. $\gr(ab) = \gr(a)\gr(b)$,
\item compatible with the differential, i.e. $\gr(\bdy a) = \lambda^{-1}\gr(a)$.
\end{itemize}
\end{defn}
In this notation, a $\Z$-grading is a grading by $(\Z, 1)$. In this paper, we discuss a $\Z/2$-grading, i.e. a grading by $(\Z/2, 1)$.  

\begin{defn}
Let  $A$ be a differential algebra graded by $(G, \lambda)$, and let $S$ be a set with a right $G$ action.  A \emph{right differential  $A$-module graded by $S$} is a right differential $A$-module $M$  with a grading $\gr$ by $S$ (as a set), i.e. a decomposition $M\cong\bigoplus_{s\in S} M_s$, so that for homogeneous $a\in A$ and $x\in M$
\begin{itemize}
\item $\gr(xa) = \gr(x)\gr(a)$,
\item $\gr(\bdy x) = \gr(x)\lambda^{-1}$.
\end{itemize}
  Similarly, a right $S$-graded $\ainf$-module over $A$ is a right $\ainf$-module $M$ over $A$ whose underlying module is graded by $S$, such that for homogeneous $x\in M$ and $a_i\in A$, we have 
  $$\gr(m_{i+1}(x\otimes a_1\otimes \cdots\otimes a_i)) = \gr(x)\gr(a_1)\cdots\gr(a_i).$$
 \end{defn}
  
 If $G$ acts transitively on $S$, we can identify $S$ with the space of right cosets of the stabilizer $G_s$ of any $s\in S$.
 
  Similarly, left modules are graded by sets with a left $G$ action (left $G$-sets).
  
   Tensor products are graded by twisted products of sets.
   \begin{defn}
   Let $A$ be a differential algebra graded by $(G, \lambda)$, let $M$ be a right $A$-module  graded by a right $G$-set $S$, and let $N$ be a left $A$-module graded by a left $G$-set $T$. The tensor product $M\otimes_{\bf k} N$ is then  graded by the set 
   $$S\times_G T := S\times T/\{(s,gt) \sim (sg, t) | s\in S, t\in T, g\in G \},$$
   and this grading descends to a grading on the chain complex $M\otimes_A N$. There is an action of $\lambda$ on $S\times_G T$ by $\lambda \cdot [s\times t]:= [s\lambda \times t] =  [s\times \lambda t]$, and the differential acts by $\lambda^{-1}$ on the gradings. 
   \end{defn} 
   
   \subsubsection{Unrefined gradings for bordered Heegaard Floer homology}
   We recall the undrefined gradings on the algebra and modules. 
   Let $\zz = (Z, {\bf a}, M, z)$ be a pointed matched circle, with $|{\bf a}| = 4k$. We recall the gradings on the algebra $\az$ and the modules over it. Let $Z' = Z\setminus \{z\}$. For $p\in {\bf a}$ and $\alpha\in H_1(Z', {\bf a})$, the \emph{multiplicity $m(\alpha, p)$ of $p$ in $\alpha$} is the average multiplicity with which $\alpha$ covers the regions adjacent to $p$. Extend $m$ bilinearly to a map $m: H_1(Z', {\bf a})\times H_0({\bf a})\to \frac 1 2 \Z$, and define the \emph{linking of $\alpha, \beta\in H_1(Z', {\bf a})$} by 
   $$L(\alpha, \beta):= m(\beta, \bdy\alpha).$$ 

Define $G'(4k)$ to be the group generated by pairs $(j, \alpha)$,  with $j\in \frac 1 2 \Z$ and $\alpha\in H_1(Z', {\bf a})$ such that $j\equiv\frac 1 4 \#\{p\in {\bf a}| m(\alpha, p) \textrm{ is a half integer}\} \mod 1$, with multiplication given by 
$$(j_1, \alpha_1)\cdot (j_2, \alpha_2) = (j_1 + j_2 + L(\alpha_1, \alpha_2), \alpha_1+\alpha_2).$$
We refer to $j$ as the \emph{Maslov component} of $(j, \alpha)$, and $\alpha$ as the \emph{$spin^c$ component} of $(j, \alpha)$. 

We define a grading $\gr'$ on $\cala(4k)$ by $G'(4k)$ as follows. For $a = (S, T, \phi)\in \cala(4k)$, define 
$$[a] = \sum_{s\in S}[s, \phi(s)]\in H_1(Z', {\bf a}), $$
i.e. $[a]$ is the sum of the intervals on $Z'$ corresponding to the strands in the strand diagram for $a$.
Also define 
$$\iota(a) = \inv(a) - m([a], S).$$
The function $\gr'$ given by 
$$\gr'(a) = (\iota(a), [a])$$
defines a grading on $\cala(4k)$ in the sense of Definition \ref{def:group_gr}, where
the distinguished central element of $G'(4k)$ is $\lambda = (1,0)$. The elements of form $a_0(\brho)$ are homogeneous with respect to this grading, and the grading $\gr'$ descends to a grading on the algebra $\az$.

With an additional choice of a base generator in each $\mathrm{spin}^c$ structure, one obtains a grading by a right $G'(4k)$-set or $G(\HH)$-set   on the right $\ainf$-module associated to a Heegaard diagram $\HH$, and similarly a grading by a left $G'(4k)$-set or $G(-\HH)$-set   on the left type $D$  structure associated to $\HH$. 

In order to define the gradings on the modules, we first grade domains. Let $\HH$ be a bordered Heegaard diagram with $\bdy \HH = \zz$. Let $B\in \pi_2(\xx, \yy)$ be a domain away from the basepoint. Define $g'(B)\in G(4k)$ by 
$$g'(B) := (-e(B) - n_{\xx}(B)- n_{\yy}(B), \bdy^{\bdy}B),$$
where $\bdy^{\bdy}B$ is the part of $\bdy B$ contained in $\bdy \HH$, viewed as an element of $H_1(Z', {\bf a})$. This grading is multiplicative under concatenation, i.e. $g'(B_1\ast B_2) = g'(B_1)g'(B_2)$. For a fixed $\xx\in \mathfrak S(\HH)$, the set $P'(\xx):=\{g'(B)| B\in \pi_2(\xx, \xx)\}$ is a subgroup of $G(4k)$. 

The $\ainf$-module $\cfahat(\HH)$ splits as a direct sum over $\mathrm{spin}^c(Y)$, where $Y$ is the bordered manifold associated to $\HH$. For each nonzero summand $\cfahat(\HH, \mathfrak s)$, fix a base generator  $\xx\in \ghs$, and define $S_A'(\HH, \mathfrak s):= P'(\xx)\backslash G'(4k)$. For a generator $\yy\in \ghs$, let $B\in \pi_2(\xx, \yy)$, and define
$$\gr'(\yy) =P'(\xx)g'(B).$$ 
The function $\gr'$ defines a grading on  $\cfahat(\HH, \mathfrak s)$ by $S_A'(\HH, \mathfrak s)$.

To define a grading on $\cfdhat(\HH)$, we must now work with  $G'(4k)$ or $G(\zz)$ where the pointed matched circle  is $\zz=-\HH$. The orientation reversing identity map $r:Z\to -Z$ induces a map $R:G'(4k)\to G'(4k)$ or $G(\zz)\to G(-\zz)$, i.e. $R(j, \alpha) = (j, r_{\ast}(\alpha))$. For nonempty $\ghs$, fix $\xx\in \ghs$ and define $S_D'(\HH, \mathfrak s):=G'(4k)/R(P'(\xx))$. Define the grading of $\yy\in \ghs$ by picking $B\in \pi_2(\xx, \yy)$ and setting
$$\gr'(\yy) = R(g'(B))R(P'(\xx)).$$
This defines a grading on the module $\cfdhat(\HH, \mathfrak s)$ by $S_D'(\HH, \mathfrak s)$.

 \subsubsection{Refined gradings for bordered Heegaard Floer homology}\label{ssec:refined}

If certain choices are made, the grading $\gr'$ descends to a grading $\gr$ on $\mathcal A(\mathcal Z, i)$ in the smaller group $G(\zz) = \{ (j, \alpha)\in G'(4k)|M_{\ast}(\alpha) = 0\}$, which is a central extension of $H_1(F)$. The group $G(\zz)$ is the Heisenberg group associated to the intersection form of $F$, and the grading function is defined as follows. Pick an arbitrary base idempotent $I(\sss)\in \mathcal A(\zz, i)$. For each other $I(\ttt)$, pick an element $\psi(\ttt)\in G'(4k)$ so that $M_{\ast}\bdy[\psi(\ttt)] = \ttt-\sss$. Define
$$\gr(I(\ttt_1)a(\brho)I(\ttt_2)) = \psi(\ttt_1)\gr'(a(\brho))\psi(\ttt_2)^{-1}.$$
The choices $\sss, \psi(\ttt)$ are called \emph{grading refinement data}. The function $\gr$ is a grading on $\cala(\zz, i)$ by $(G(\zz), \lambda)$. Different choices of grading refinement data result in gradings which are conjugate in a certain sense, see \cite[Remark 3.46]{bfh2}.

After choosing refinement data, one can grade a domain $B\in \pi_2(\xx, \yy)$ by 
$$g(B) = \psi(I_A(\xx))g'(B) \psi(I_A(\yy))^{-1}\in G(\HH).$$
For $\ghs\neq \emptyset$, fix $\xx\in \ghs$. The set $P(\xx) = \{g(B)|B\in \pi_2(\xx, \xx)\}$ is a subgroup of $G(\HH)$, and $\cfahat(\HH, \mathfrak s)$ has a grading in $S_A(\HH, \mathfrak s) = P(\xx)\backslash G(\HH)$ via
$$\gr(\yy) = P(\xx)g(B), $$ 
where $B\in \pi_2(\xx, \yy)$.

To define a refined grading on $\cfdhat(\HH)$, first fix grading refinement data $\sss_{\bdy \HH}$ and $\psi_{\bdy \HH}$  for $\cala(\bdy\HH, 0)$, so that a refined grading $g$ on domains is defined. Then $\sss_{-\bdy \HH}= [2k]\setminus \sss_{\bdy\HH}$ and $\psi_{-\bdy \HH}(\ttt) = R(\psi_{\bdy \HH}([2k]\setminus\ttt))^{-1}$ define a grading refinement on $\cala(-\bdy \HH, 0)$, the \emph{reverse of the refinement on $\cala(\HH, 0)$}. For nonempty $\ghs$, fix $\xx\in \ghs$, and define $S_D(\HH, \mathfrak s) := G(-\bdy\HH)/R(P(\xx))$.  Then $\cfdhat(\HH, \mathfrak s)$ has a grading in $S_D(\HH, \mathfrak s)$ defined on generators $\yy\in\ghs$ by picking $B\in \pi_2(\xx, \yy)$ and setting 
$$\gr(\yy) = R(g(B))R(P(\xx)).$$

   \subsubsection{Tensor products}

For bordered Heegaard diagrams $\HH_1$ and $\HH_2$ that agree along the boundary, the tensor product $\cfahat(\HH_1, \mathfrak s_1)\boxtimes \cfdhat(\HH_2, \mathfrak s_2)$ has a grading $\gr^{\boxtimes}$ in the double-coset space
$$P_1(\xx_1)\backslash G(\zz)/R(P_2(\xx_1)),$$
where the base generators $\xx_1$ and $\xx_2$ are chosen to occupy complementary $\alpha$-arcs, and the grading refinement for $\cfahat(\HH_1, \mathfrak s_1)$ is chosen the same  as the grading refinement for $\cfdhat(\HH_2, \mathfrak s_2)$. Here $P_i(\xx_i)$  denotes the image of $\pi_2(\xx_i, \xx_i)$ in $G(\HH_i)$.

There is a homotopy equivalence 
$$\Phi:\cfahat(\mathcal H_1, \mathfrak s_1) \boxtimes \cfdhat(\mathcal H_2, \mathfrak s_2)\to  \bigoplus_{\substack{\mathfrak s\in \mathrm{spin}^c(Y)   \\ \mathfrak s|{Y_i} = \mathfrak s_i, i = 1,2 }}\cfhat(\HH, \mathfrak s)$$
which respects the gradings in the following sense. If  $u\boxtimes x$ and $v\boxtimes y$ are in the same $\mathrm{spin}^c$ structure $\mathfrak s$, and are homogeneous with gradings  related by 
$$\gr^{\boxtimes}(u\boxtimes x) = \lambda^t \gr^{\boxtimes}(v\boxtimes y),$$
then $\Phi(u\boxtimes x)$  and $\Phi(v\boxtimes y)$ are homogeneous with respect to the Maslov grading, and their relative Maslov grading is $t$ modulo $\mathrm{div}(c_1(\mathfrak s))$. 
See \cite[Theorem 10.42]{bfh2},

\section{A $\Z/2$ grading}\label{z2sec}

Instead of working with the grading groups $G'(4k)$ and $G(\zz)$, we will define a homological $\Z/2$ grading $m$ on the algebra $\az$, and study the \grg of $\Z/2$-graded modules over $\az$. The reason we prefer this is because we would like to easily relate the Euler characteristic of  $\cfahat$ or $\cfdhat$ of bordered $3$-manifolds in $K_0(\az)$ to other familiar $3$-manifold invariants.

Let  $\zz = (Z, {\bf{a}}, M, z)$ be a a pointed matched circle with associated surface $F = F(\zz)$ of genus $k$.  For each $i\in [2k]$, let $\rho_i$ be the Reeb chord connecting the two points in $M^{-1}(i)$.  Choose grading refinement data, i.e. a base idempotent $I(\sss)$ and an element $\psi(\ttt)\in G'(4k)$ for each $\ttt$, as in Section \ref{ssec:refined}, to get a  grading $\gr$ on $\az$ with values in $G(\zz)$.  For any $\rho_i$ and any refinement data, $a(\rho_i)$ is homogeneous with respect to the refined grading, so let $g_i = \gr(a(\rho_i))$ for $i\in [2k]$.
The set $\{[a(\rho_i)]| i \in [2k]\}$ generates $H_1(F; \Z)$, and $\lambda$ generates the center of $G(\zz)$, so $\{\lambda, g_1, \ldots, g_{2k}\}$ is a generating set for $G(\zz)$. Since the commutators of $G(\zz)$ are precisely the even powers of $\lambda$, we see that the abelianization of $G(\zz)$ is 
$$G(\zz)^{Ab} = \Z/2\times \Z^{2k}.$$
Thus, given any abelian group $H$, any map from  $\{\lambda, g_1, \ldots, g_{2k}\}$ to $\Z/2\times H$ sending $\lambda$ into $\Z/2\times \{0\}$
determines a homomorphism $h$ from $G(\zz)$ to $\Z/2\times H$. Moreover, if  $h(\lambda)=(1,0)$, the composition $h\circ \gr: \az\to \Z/2\times H$ is a grading, with the differential lowering the grading by $(1, 0)$. 

Here we choose $H$ to be the trivial group, and define a homomorphism 
\begin{align*}
f: G(\zz) &\to \Z/2\\
\lambda &\mapsto 1\\
g_i &\mapsto 1
\end{align*}

\begin{defn}
The $\Z/2$ grading of an element $a\in \az$ is defined as  $m(a) := f\circ \gr (a)$.
\end{defn}
The discussion above shows that $m$ is a differential grading on $\az$ (with $\bdy$ lowering the grading by $1$ mod $2$).

With this $\Z/2$ grading on $\az$, we are ready to prove Theorems \ref{intro_k0} and \ref{intro_tensor}, which are statements for differential graded modules  (where the grading is by $\Z/2$). However, in order to define an invariant in $K_0(\az)$ of bordered manifolds, we need to define a $\Z/2$ grading on $\cfahat\hs$ and $\cfdhat\hs$.

First we define a $\Z/2$ grading on domains, using the refined grading $g$ by $G(\zz)$.

\begin{defn} 
For any $x, y \in \mathfrak S(\mathcal H)$ and $B\in \pi_2(x,y)$, define $m(B)\in \Z/2$ by
$$m(B):= f\circ g(B).$$
\end{defn}

\begin{theorem}\label{domains}
For periodic domains, the map $m$ is independent of the grading refinement data. In fact, if $B$ is a periodic domain, $m(B)=0$.
\end{theorem}

Before we prove the theorem, we make an observation about the refined grading on the generators $a(\rho_i)$, and then give an alternate definition of $m$.
\begin{lemma}\label{gi}
Reducing the Maslov component modulo $2$, the refined grading on $a(\rho_i)$ is given by 
\begin{displaymath}
g_i = \left\{ \begin{array}{ll}
(-\frac{1}{2}; \left[a(\rho_i)\right]) \quad & \textrm{if  $i\in {\bf s}$}\vspace{.2cm}\\ 
 (\frac{1}{2}; \left[a(\rho_i)\right]) & \textrm{if $i\notin {\bf s}$.}
\end{array} \right.
\end{displaymath}
\end{lemma}
\begin{proof}
If $i\in \sss$, then $I(\sss)a(\rho_i)I(\sss)$ is nonzero, so $g_i = \gr'(a(\rho_i)) = (-\frac{1}{2}; \left[a(\rho_i)\right])$.
If $i\notin \sss$, then choose some $j\in \sss$ and define $\ttt := (\sss\setminus j) \cup i$. Then $I(\ttt)a(\rho_i)I(\ttt)\neq 0$, so 
\begin{align*}
g_i &= \psi(\ttt) \gr'(a(\rho_i))\psi(\ttt)^{-1}\\
&= (-\frac{1}{2}+ L(\left[\psi(\ttt)\right], \left[a(\rho_i)\right]) + L(\left[\psi(\ttt)\right], \left[\psi(\ttt)^{-1}\right]) + L(    \left[a(\rho_i)\right] , \left[\psi(\ttt)^{-1}\right]); \left[a(\rho_i)\right])\\
&= (-\frac{1}{2}+ L(\left[\psi(\ttt)\right], \left[a(\rho_i)\right]) + L(    \left[a(\rho_i)\right] , \left[\psi(\ttt)^{-1}\right]); \left[a(\rho_i)\right])\\
&= (-\frac{1}{2}+ L(\left[\psi(\ttt)\right], \left[a(\rho_i)\right]) - L(    \left[a(\rho_i)\right] , \left[\psi(\ttt)\right]); \left[a(\rho_i)\right])\\
&= (-\frac{1}{2}+ 2 L(\left[\psi(\ttt)\right], \left[a(\rho_i)\right]); \left[a(\rho_i)\right])
\end{align*}
Now, write  $\bdy\left[\psi(\ttt)\right] = \sum_{p=1}^{4k} n_pa_p$.
Since $M_{\ast}\bdy\left[\psi(\ttt)\right] = \ttt-\sss = j-i$, and $M^{-1}(i) = \{\rho_i^+, \rho_i^-\}$, it must be that $n_{\rho_i^+} + n_{\rho_i^-} = 1$. Then 
\begin{align*}
L(\left[\psi(\ttt)\right], \left[a(\rho_i)\right]) &= m(\left[a(\rho_i)\right], \left[\psi(\ttt)\right])\\
&=  m(\left[a(\rho_i)\right],  \sum_{p=1}^{4k} n_pa_p)\\
&=   \sum_{p=1}^{4k} n_p m(\left[a(\rho_i)\right], a_p)\\
&\equiv n_{\rho_i^+} m(\left[a(\rho_i)\right], \rho_i^+)+ n_{\rho_i^-} m(\left[a(\rho_i)\right], \rho_i^-)\\
&= \frac{1}{2}n_{\rho_i^+}+ \frac{1}{2}n_{\rho_i^-}\\
&= \frac{1}{2}\text{ mod } 1.
\end{align*}
It follows that $ 2 L(\left[\psi(\ttt)\right], \left[a(\rho_i)\right])\equiv 1 \text{ mod } 2$, so  $g_i =  (\frac{1}{2}; \left[a(\rho_i)\right])$.
\end{proof}

Using $\left[a(\rho_i)\right]$ as a standard basis for $H_1(F; \Z)$, if $\alpha = \sum_{i=1}^{2k}h_i\left[a(\rho_i)\right]$, we will write $(j; \alpha)$ as $(j; h_1, \ldots, h_{2k})$.

\begin{proposition} 
Let $\delta_{ij}:=\left[a(\rho_i)\right]\cap \left[a(\rho_j)\right]=  L(\rho_i, \rho_j)$, and define a map $\fs$ from $G(\zz)$ to $\Z/2$ by
$$\fs(j; h_1, \ldots, h_{2k}) = j - \frac{1}{2}\sum_{i\in \sss}h_i + \frac{1}{2}\sum_{i\notin \sss}h_i+ \sum_{i_1<i_2}h_{i_1} h_{i_2} \delta_{i_1 i_2}. $$
Then $\fs$ agrees with the homomorphism $f$.
\end{proposition}

\begin{proof}
Note that  $ \left[a(\rho_{i_1})\right]\cap \left[a(\rho_{i_2})\right] = -  \left[a(\rho_{i_2})\right]\cap \left[a(\rho_{i_1})\right]\in \Z$ implies  $ \delta_{i_1 i_2}\equiv \delta_{i_2 i_1} \textrm{ mod } 2$.

Given $a = (j_a; h_1^a, \ldots, h_{2k}^a)$ and $b = (j_b; h_1^b, \ldots, h_{2k}^b)$, 
\begin{align*}
\fs(ab) &= \fs(j_a + j_b + L(\left[a\right], \left[b\right]); h_1^a+ h_1^b, \ldots, h_{2k}^a + h_{2k}^b)\\
&= j_a + j_b + L(\left[a\right], \left[b\right]) - \frac{1}{2}\sum_{i\in \sss}(h_i^a+ h_i^b)+ \frac{1}{2}\sum_{i\notin \sss}(h_i^a+ h_i^b) \\
&\hspace{1.7in}+ \sum_{i<j}(h_i^a+ h_i^b)(h_j^a+ h_j^b) \delta_{ij}\\
&= j_a + j_b + L(\left[a\right], \left[b\right]) - \frac{1}{2}\sum_{i\in \sss}(h_i^a+ h_i^b)+ \frac{1}{2}\sum_{i\notin \sss}(h_i^a+ h_i^b) \\
&\hspace{1.7in}+ \sum_{i<j}h_i^ah_j^a \delta_{ij}+ \sum_{i<j}h_i^ah_j^b \delta_{ij} + \sum_{i<j}h_j^ah_i^b \delta_{ij}+ \sum_{i<j}h_i^bh_j^b \delta_{ij}\\ 
&= \fs(a)+ \fs(b) +  L(\left[a\right], \left[b\right])+ \sum_{i<j}h_i^ah_j^b \delta_{ij} + \sum_{i>j}h_i^ah_j^b \delta_{ij}\\
&= \fs(a)+ \fs(b) +  L(\left[a\right], \left[b\right])+ \sum_{i\neq j}h_i^ah_j^b \delta_{ij} \\
&= \fs(a)+ \fs(b)
\end{align*}
The last equality  is true since $L(\left[a\right], \left[b\right])$ is an integer, and $\sum_{i\neq j}h_i^ah_j^b \delta_{ij} \equiv L(\left[a\right], \left[b\right])\text{ mod }2$. Thus, $\fs$ is a homomorphism.

We evaluate 
\begin{displaymath}
\fs(g_i) = \left\{ \begin{array}{ll}
-\frac{1}{2} - \frac{1}{2}\cdot 1+ \frac{1}{2}\cdot 0+ 0 = 1 \textrm{ mod }2  \quad & \textrm{if  $i\in {\bf s}$}\vspace{.2cm}\\ 
\frac{1}{2} - \frac{1}{2}\cdot 0+ \frac{1}{2}\cdot 1+ 0= 1  \textrm{ mod }2 & \textrm{if $i\notin {\bf s}$}
\end{array} \right.
\end{displaymath}
and $\fs(\lambda) = 1$, so $\fs$ agrees with $f$ on the generators $g_i$ and $\lambda$. Thus, $\fs\equiv f$.
\end{proof}

\begin{proof}[Proof of Theorem \ref{domains}]

We may assume that $B\in \pi_2(\xx,\xx)$ for some $\xx$ with $I_A(\xx) = I(\sss)$, since if $B_1\in \pi_2(\xx,\xx)$ and $I_A(\xx)\neq I(\sss)$, we can choose a generator $\yy$ with $I_A(\yy) = I(\sss)$, and a domain $B_2\in \pi_2(\yy,\xx)$, so that $B:=B_2\ast B_1\ast(-B_2)\in \pi_2(\yy,\yy)$, and we see that
\begin{align*}
\fs(g(B)) &= \fs(g(B_2\ast B_1\ast(-B_2)))\\
&= \fs(g(B_2)g(B_1)g(-B_2))\\
&=  \fs(g(B_2)g(B_1)g(B_2)^{-1})\\
&=  \fs(g(B_2))+ \fs(g(B_1)) - \fs (g(B_2))\\
&= \fs(g(B_1)).
\end{align*}

We construct a surface $F$ as in the proof of \cite[Lemma 10.3]{bfh2}. We follow the notation of that proof without explaining it for the next few paragraphs, so we advise the reader to get familiar with it before proceeding. 

If necessary, perform a $\beta$-curve isotopy as in Figure \ref{regions} to arrange that distinct segments of $\bdy \Sigma$ lie in distinct regions of the Heegaard diagram, and label the regions $R_0, R_1, \ldots, R_{4k-1}$, beginning at the basepoint $z$, and following the orientation of $\bdy \Sigma$.  Let $\delta_i = m(R_i) - m(R_{i-1})$. 
      \vskip .2 cm

\begin{figure}[h]
 \centering
       \labellist
                \pinlabel $\cdots$ at 116 12
       \pinlabel $\to$ at 230 20
              \pinlabel $R_0$ at 267 12
       \pinlabel $R_1$ at 296 12
        \pinlabel $R_2$ at 328 12
         \pinlabel $\cdots$ at 380 12
         \pinlabel $R_{4k-1}$ at 426 12
                  \endlabellist
           
       \includegraphics[scale=.9]{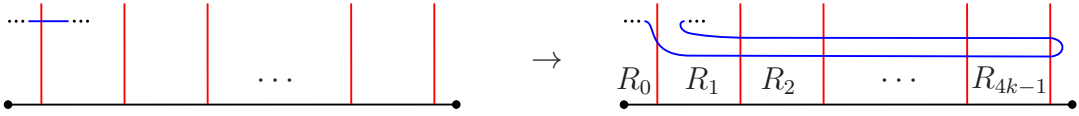}
       \vskip .2 cm
       \caption{An isotopy ensuring that there are $4k$ distinct regions at $\bdy \Sigma$. On the left, a collar neighborhood of $\bdy\Sigma$ and the closest intersection point to $a_1$ along $\alpha_1$. On the right, a finger move along $\bdy\Sigma$ of the $\beta$-curve near that intersection point.}
       \label{regions}
\end{figure}

If $x\in \xx$, then $C(x) = 0$, so the lift $\Phi^{-1}D(x)$ of a neighborhood of $x$ is a union of disks and $h(x)$ half-disks.
Observe that $e(\left[\Sigma\right]) + 2n_\xx (\left[\Sigma\right]) = 1$, so 
\begin{align*}
e(B)+ 2n_\xx(B) &\equiv e(\widetilde B)+ 2n_\xx(\widetilde B)- l(e(\left[\Sigma\right]) + 2n_\xx (\left[\Sigma\right]))\\
&\equiv e(F)+ 2\sum_{x\in \xx}n_{x}(F) - l
\end{align*}
\begin{align*}
\phantom{e(B)+ 2n_\xx(B)} &\equiv \chi(F) - \frac{1}{4}\sum |\delta_i| + \sum_{x\in \xx}h(x) - l\\
&\equiv  \#(\bdy F) - \frac{1}{4}\sum |\delta_i| + \sum_{x\in \xx}h(x) - l
\end{align*}
The boundary of any half-disk lies above an $\alpha$- or a $\beta$-circle, or an $\alpha$-arc. Since each $\alpha$ or $\beta$-circle is occupied exactly once by an $x\in \xx$, we can cancel in pairs all half disks with boundaries lying above $\alpha$- or $\beta$-circles with the components of $\bdy F$ that are lifts of the same $\alpha$- or $\beta$-circles. We are left only with half disks with boundary projecting to an $\alpha$-arc, and with components of $\bdy F$ whose projection intersects $\bdy \Sigma$. Let the number of these connected components of $\bdy F$ be $t$, and let $h^a(x)$ be the number of half disks at $x$ with boundary projecting to an $\alpha$-arc. Then 
$$e(B)+ 2n_\xx(B) \equiv t - \frac{1}{4}\sum |\delta_i| + \sum_{x\in \xx}h^a(x) - l.$$
Now, 
\begin{align*}
|\delta_i| &= \# \textrm{ corners of $F$ above } a_i \\
&= |\# \textrm{ preimages of } \alpha_{M(i)} \textrm{ in }\bdy F|\\
&= |\textrm{ multiplicity of } \alpha_{M(i)} \textrm{ in } \bdy B|\\
&= |h_{M(i)}|,
\end{align*}
and
$$
\sum_{x\in \xx}h^a(x) = \sum_{x\in \xx}  |\textrm{ multiplicity of the $\alpha$-arc occupied by $x$ in } \bdy B|
= \sum_{i\in \sss}|h_i|, 
$$
so 
\begin{align*}
e(B)+ 2n_\xx(B) &\equiv t-l - \frac{1}{4}\sum_{i=1}^{4k} |h_{M(i)}| + \sum_{i\in \sss}|h_i| \\
&\equiv t-l - \frac{1}{2}\sum_{i=1}^{2k} |h_i| + \sum_{i\in \sss}|h_i| \\
&\equiv  t -l- \frac{1}{2}\sum_{i\notin \sss} |h_i| +\frac{1}{2} \sum_{i\in \sss}|h_i|.
\end{align*}
Thus,
\begin{align*}
\fs(g(B)) &=  -e(B)- 2n_\xx(B) - \frac{1}{2}\sum_{i\in \sss}h_i + \frac{1}{2}\sum_{i\notin \sss}h_i+ \sum_{i_1<i_2}h_{i_1} h_{i_2} \delta_{i_1 i_2}\\
&= -t +l+ \frac{1}{2}\sum_{i\notin \sss} (|h_i|+h_i) -\frac{1}{2} \sum_{i\in \sss}(|h_i|+h_i) + \sum_{i_1<i_2}h_{i_1} h_{i_2} \delta_{i_1 i_2}\\
&= -t +l+ \sum_{i\notin \sss, h_i>0} h_i -\sum_{i\in \sss, h_i>0}h_i + \sum_{i_1<i_2}h_{i_1} h_{i_2} \delta_{i_1 i_2}\\
&= -t +l+ \sum_{h_i>0} h_i + \sum_{i_1<i_2}h_{i_1} h_{i_2} \delta_{i_1 i_2}.
\end{align*}

Since lifts of adjacent regions are identified along $\alpha$-arcs in pairs, starting at the highest index and going down, the lift of $\bdy^\bdy \widetilde B$ consists of positively oriented arcs of $\bdy \Sigma$ such that any two are either disjoint or nested, but never interleaved or abutting. In other words, we can represent the result of this identification geometrically on an annulus lying above $\zz$ by layers of horizontal arcs, so that the lowest layer consists of all $\bdy^\bdy R_i^{(j)}$ of highest index $(j)$, after gluing, the second layer from the bottom contains the second highest indices $(j)$, and so on. In this representation, each arc not in the lowest layer projects to a subset of the projection of the arc that it lies over. Since the identification along $\beta$-arcs is from lowest to highest index, following the regions along the higher multiplicity side of an $\alpha$-arc at $\bdy\Sigma$ shows that the boundaries of the set of arcs are matched along $\alpha$-arcs in order from the highest level (i.e. lowest index) to the lowest possible, to form the $t$ circles of $\bdy F$ that contain $\alpha$-arcs.  See, for example, Figure \ref{tower}.

\begin{figure}[h]
 \centering
       \labellist
       \pinlabel $z$ at -5 1
       \pinlabel $z$ at 280 1
       \pinlabel $z$ at 139 69
                  \endlabellist
       \includegraphics[scale=.9]{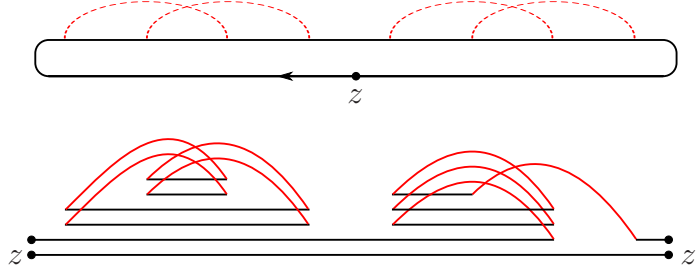}
       \vskip .2 cm
       \caption{Top: The genus $2$ split circle. Bottom: The layers representation of  the boundary components of $F$ which contain parts of $\bdy \Sigma$, in the case when $B$ is
       a domain with boundary $(2, 2, 3, -1)$ and after  two copies of $\Sigma$ have been added to $B$. Here $l=2$, $l_0=1$.}
       \label{tower}
\end{figure}

Rotating $90^{\circ}$ counterclockwise, we can draw $\zz$ in the plane  as a vertical line segment oriented upwards with ends identified with the basepoint $z$, and we can draw the annulus as a rectangle, so that the top and bottom edges are identified and project to $z$, and so that when we endow all $\bdy \Sigma$-arcs and $\alpha$-arcs with an orientation arising from the orientation of $F$, all $\bdy\Sigma$-arcs are oriented upwards.

Note that for a matched pair $(i,j)$ with $i<j$, $h_{M(i)} = \delta_i = -\delta_j$, i.e.  for $i\in \left[2k\right]$, $h_i = \delta_{\rho_i^-} = -\delta_{\rho_i^+}$. If $h_i<0$, then $\delta_{\rho_i^-}<0$, so there are $|h_i|$ $\bdy\Sigma$-arcs ending at $\rho_i^-$, hence $|h_i|$ $\alpha$-arcs starting at $\rho_i^-$ and ending at $\rho_i^+$. In other words, all $|h_i|$  copies of $\alpha_i$ are oriented upwards when $h_i<0$. Similarly, when $h_i>0$, there are $h_i$ $\bdy\Sigma$-arcs starting at $\rho_i^-$, hence $h_i$ $\alpha$-arcs ending at $\rho_i^-$ and starting at $\rho_i^+$, i.e. all $h_i$ copies of $\alpha_i$ are oriented downwards when $h_i<0$. (see Figure \ref{tower_braid}). 

For $h_i<0$, we can draw the copies of $\alpha_i$ as parallel arcs open to the right, curved so that they do not intersect any $\bdy\Sigma$-arcs. For $h_i>0$, we draw the copies of $\alpha_i$ similarly as  arcs starting at $\rho_i^+$ moving upwards, passing through the top horizontal line, continuing to move up from the bottom  to the copies of $\rho_i^-$, also in a way as to not intersect $\bdy\Sigma$-arcs. All arcs now move strictly upwards, and in this way we represent the relevant boundary components of $F$ as a closed braid (where we don't care about the sign of crossings). Note that a copy of $\alpha_i$ and a copy of $\alpha_j$ cross an even number of times if $\delta_{ij} = 0$, and an odd number of times if $\delta_{ij} = 1$. It follows that if the braid is given by a permutation $\sigma$,
\begin{align*}
t &= \#\textrm{ boundary components of $F$ intersecting $\bdy\Sigma$ }\\
&= \#\textrm{ components in the closure of the braid}\\
&= \#\textrm{cycles in $\sigma$, including cycles of length $1$}\\
&\equiv \#\textrm{ involutions of $\sigma$ } + \textrm{ length of $\sigma$ }\\
&\equiv \#\textrm{ crossings in the braid } + \#\textrm{ strands in the braid }\\
&\equiv \sum_{i_1<i_2}h_{i_1} h_{i_2} \delta_{i_1 i_2} + (\sum_{h_i>0} h_i + l).
\end{align*}
Hence, $\fs(g(B))= 0$. \qedhere

\begin{figure}[h]
 \centering
       \includegraphics[scale=.75, angle=90]{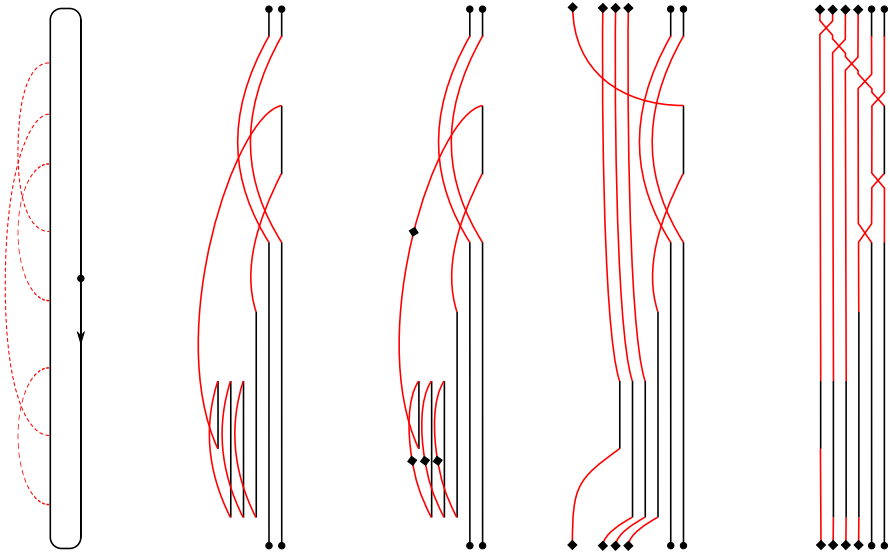}
       \vskip .2 cm
       \caption{Obtaining a braid representation of the boundary components of $F$ that intersect $\bdy\Sigma$. In this example $(h_1, h_2, h_3, h_4) = (3, 1, -1, -2)$.}
       \label{tower_braid}
\end{figure}
\end{proof}

Note that if $B$ is periodic, then $\fs(R(g(B)))=0$ too. 

Now we define $m$ as a relative grading on the modules.  
\begin{defn}
Suppose $\ghs\neq 0$. Pick $\xx\in \ghs$ and  define $m:\cfahat\hs\to \Z/2$ by $m(\xx) = 0$ and
$$m(\yy) = m(\xx)+m(B),$$
where $\yy\in \ghs$ and $B\in \pi_2(\xx, \yy)$.
Similarly,  define $m:\cfdhat\hs\to \Z/2$ by $m(\xx) = 0$ and
$$m(\yy) = m(\xx)+\fs(R(g(B))),$$
where $\yy\in \ghs$ and $B\in \pi_2(\xx, \yy)$. 
\end{defn}
This is well defined since it factors through the set gradings from Section \ref{ssec:refined}.

In some cases there is a natural choice of a base generator for each $\mathrm{spin}^c$ structure so that $m$ agrees with the absolute Maslov grading after gluing. We do not discuss these choices here.

\section{The Grothendieck group of $\Z/2$-graded $\az$-modules}\label{k0sec}

Let $\mathcal C$ be a triangulated category \cite{ve}. The \emph{Grothendieck group $K_0(\mathcal C)$} is defined as the abelian group with generators $[P]$ for each isomorphism class $P\in \mathcal C$
and relations  $[P] = [P']+[P'']$ for each distinguished triangle  $P'\to P\to  P''\to P'[1]$.

\example The most familiar and basic example is the homotopy category of $\Z$-graded chain complexes over $\Z$ (the translation functor $[1]$ simply shifts the complex, and the distinguished triangles are the mapping cones). The reader can verify that the Grothendieck group is $\Z$, via $[C] = \chi(C)$.

As mentioned in Section $1$, specific examples have been studied in low dimensional topology, namely, the image in the corresponding Grothendieck groups of the homology theories with $\Z$ or $\Z/2$ coefficients $\sfh$, $\kh$, $\hf$, $\hfk$ \cite{dsfh, kh1, osz14, hfk}.

We proceed to define our main object of interest, the Grothendieck group  of a differential graded algebra. 
Given a differential graded algebra $A$, we recall the definition of $K_0(A)$ (see \cite{kh2}, for example). Let $\mathcal K(A)$ be the homotopy category of modules over $A$, $\mathcal{KP}(A)$  the full subcategory of projective modules, and $\mathcal P(A)\subset \mathcal{KP}(A)$  the full subcategory of compact projective modules, which is a triangulated category.  

\begin{defn}\label{k0def}
Given an algebra $A$ with differential grading by $\Z$ or $\Z/2$, we define $K_0(A)$  to be the Grothendieck group  of  the category $\mathcal P(A)$. It has
\begin{itemize}
\item generators: $[P]$ over all compact projective differential graded $A$-modules $P$
\item relations: \begin{enumerate}\item $[P_2] = [P_1]+[P_3]$ whenever $P_1 \to P_2 \to P_3$ is a distinguished triangle 
\item $ [P[1]] = -[P]$, where $[1]$ is the grading  shift by $1$
\item If we also introduce an additional $\frac{1}{2}\Z$-grading along with the relation
$[P\{k\}] = t^k [P]$, where $\{k\}$ denotes the grading shift up by $k\in \frac{1}{2}\Z$, we get a $\Z[t^{\frac{1}{2}}, t^{-\frac{1}{2}}]$-module for $K_0$.
 \end{enumerate}
\end{itemize}
\end{defn}

Note that relation (3) is usually seen in reference to a $\Z$-grading, but we find it more convenient in the bordered Floer homological context to work with a $\frac{1}{2}\Z$-grading, which we define and study in Section \ref{alexsec}.

By \cite{bimod}, there is an equivalence of categories of bounded left type $D$ structures and right $\ainf$-modules, and along with \cite[Corollary 2.3.25]{bimod}, we see that we can work either in the homotopy category of finitely generated projective modules or finitely generated bounded type $D$ structures. We will prove Theorem \ref{intro_k0} for left type $D$ structures. The same result for right $\ainf$-modules follows by the equivalence of categories.

For the rest of this section, we will often write $\cala$ for $\az$.

Recall that a type $D$ structure $N$ over $\cala$ is called \emph{bounded} if for all $x\in N$ there is some integer $n$, so that for all $i\geq n$, $\delta_i(x) = 0$. Note that for a finitely generated $N$ we can find a universal $n$, so that for all $x$ and all $i\geq n$, $\delta_i(x) = 0$. By \cite[Proposition 2.3.10]{bimod}, every type $D$ structure over $\cala$ is homotopy equivalent to a bounded one, by tensoring it with a left and right bounded $\mathit{AD}$ identity bimodule. Furthermore, finiteness is preserved, since the identity module   in the proof of \cite[Proposition 2.3.10]{bimod} is finitely generated. 
In the homotopy category, homotopy equivalent type $D$ structures map to the  same symbol in  $K_0$, so it is enough to study finitely generated bounded type $D$ structures.

\note In the special case of $\cfdhat$ of a 3-manifold, the bounded type $D$ structures are the ones coming from admissible Heegaard diagrams. \\

\begin{defn}
Let $M$ be a right $\ainf$-module over $\cala$. We say that $x\in M$ is \emph{$\mathcal I$-homogeneous} if there is a unique indecomposable idempotent, which we denote by $I_A(x)$, that acts  on $x$ by the identity, and all other indecomposable idempotents act trivially on $x$. We denote by $o(x)$ the subset of $[2k]$ for which $I_A(x)=I(o(x))$.
\end{defn}

Recall that a left type $D$ structure over $\cala$ is an $\mathcal I$-module equipped with a type $D$ structure map, so we can talk about an action on $N$ by $\mathcal I$ too. 
\begin{defn}
Let $N$ be a left type $D$ structure over $\cala$.
We say that $x\in N$ is \emph{$\mathcal I$-homogeneous} if there is a unique indecomposable idempotent, which we denote by $I_D(x)$, that acts  on $x$ by the identity, and all other indecomposable idempotents act trivially on $x$. We denote by $\bar o(x)$ the subset of $[2k]$ for which $I_A(x)=I(\bar o(x))$.
\end{defn}

From now on we work with type $D$ structures. 

\begin{lemma}\label{lem:ihom}
Let $N$ be a finitely generated $\Z/2$-graded type $D$ structure over $\cala$. Then $N$ has a finite set of generators that are homogeneous with respect to the grading and $\mathcal I$-homogeneous. 
\end{lemma}

\begin{proof}
This is a direct consequence of the fact that $\mathcal I\cong \Z/2^{2^{2k}}$, but we spell it out anyway. 
By definition, there is a finite generating set $\mathfrak S$ for  $N$ that is homogeneous with respect to the grading. Let $x\in \mathfrak S$. Then $I(\sss)x$ is $\mathcal I$-homogeneous, since the indecomposable idempotents are orthogonal, i.e. $I(\sss)I(\ttt)=0$ whenever $\sss\neq \ttt$. Clearly, $I(\sss)x$ is also homogeneous with respect to the grading. Last, observe that $x = {\bf I}x = \sum_{\sss\in [2k]}I(\sss)x$, so the set 
$\{I(\sss)x | x\in \mathfrak S, \sss\in [2k]\}$
generates  $N$. 
\end{proof}

Given a pointed matched circle $\zz$ for a surface $F$, let $\cala :=\az$. Recall that  in Section \ref{z2sec} we defined $\rho_i$ as the Reeb chord connecting the two points in $M^{-1}(i)$. The elements $[a(\rho_i)]$ generate $H_1(F;\Z)$. Denote the initial endpoint of $\rho_i$ by $\rho_i^-$ and the final endpoint  by $\rho_i^+$, and order the basis $\{[a(\rho_i)]| i\in [2k]\}$ so that $[a(\rho_i)]<[a(\rho_j)]$ if and only if $\rho_i^-$ comes before $\rho_j^-$ when we follow the orientation of the circle starting at $z$. In other words, we order the basis of $H_1(F; \Z)$ according to the order of the initial endpoints of the corresponding Reeb chords, where this order is induced by the orientation of the circle. Call the ordered basis $a_1<a_2<\cdots<a_{2k}$.

Given a set $\sss \subset [2k]$, let $J(\sss)$ be the multi-index $(j_1, \ldots, j_n)$, so  that $1\leq j_1<\ldots < j_n\leq 2k$ and $\{j_1, \ldots, j_n\} = \sss$. Given a multi-index $J = (j_1, \ldots, j_n)$ of increasing numbers as above, define $a_J = a_{j_1}\wedge\ldots \wedge a_{j_n}$. We will use the shortcut notation $a_{\sss}:= a_{J(\sss)}$. Note that $a_{\sss}$ form a basis for $\Lambda^\ast(H_1(F; \Z))$. 

Suppose $M$ is a right $\ainf$-module over $\cala$ with a set of  $\mathcal I$-homogeneous generators $\mathfrak S(M)$. 
Define a function $h:\mathfrak S(M)\to \Lambda^\ast H_1(F; \Z)$ by $h(x) =  a_{o(x)}$. Similarly, if $N$ is a left type $D$ structure with a set of $\mathcal I$-homogeneous generators $\mathfrak S(N)$, define $h:\mathfrak S(N)\to \Lambda^\ast H_1(F; \Z)$ by $h(x) = a_{\bar o(x)}$. Recall that for the structure $\cfahat$ associated to a Heegaard diagram, $I_A(x)$ is the idempotent corresponding to the $\alpha$ arcs occupied by $x$, and for $\cfdhat$, $I_D(x)$ is the idempotent corresponding to the unoccupied $\alpha$ arcs, and in either case $h(x)\in \Lambda^k(H_1(F; \Z))$.

Below is the full version of Theorem \ref{intro_k0} and its proof. 

\begin{theorem}\label{k0full}
Let $\zz$ be a pointed matched  circle with associated surface $F$ of genus $k$. The Grothendieck group of the category of finitely generated $\Z/2$-graded left type $D$ structures over $\az$  is equivalent to that of finitely generated right $\cala_\infty$-modules over $\az$ and is given by
$$K_0(\az) = \Lambda^\ast(H_1(F; \Z)).$$ 
Moreover, if $M$ is a (finitely generated) left type $D$ structure or a right $\cala_\infty$-module over $\az$ with $\Z/2$ grading $m$, then its image in this group can be computed by 
$$[M] = \sum _{x \in \mathfrak S(M)} (-1)^{m(x)}h(x),$$
 where $\mathfrak S(M)$ is a set of homogeneous generators and
  $h(x)$, as defined above,  is the wedge  in $\Lambda^\ast(H_1(F; \Z))$ of generators of $H_1(F; \Z)$ given by the set of matched points in $\zz$ corresponding to $I_A(x)$ for $\ainf$-modules, or to $I_D(x)$ for type $D$ structures (with order induced by the orientation of the circle). In other words, $[M]$ counts generators of $M$ in each primitive idempotent. 
\end{theorem}

\begin{proof}

Suppose that 
$M$ is a finitely generated bounded type $D$ structure. By Lemma \ref{lem:ihom}, $M$ has a finite  set of generators that are homogeneous with respect to the grading and $\mathcal I$-homogeneous. 
We show that then $[M]$ is a linear combination of symbols of elementary projectives, i.e., type $D$ structures of form $\cala I(\sss)$. We use induction on the size of the generating set. Clearly, if $M$ has only one generator, it is an elementary projective, since, because of the boundedness condition, the  differential must be zero. 
Otherwise, if we fix an ordering of the generators $x_1, \ldots, x_n$, we can represent the type $D$ map $\delta$ by the matrix formed by the coefficients of $\delta x_i = \sum_j a_{ij} x_j$. Observe that $M$ being bounded is equivalent to the existence of an ordering on the generators that yields  an upper triangular differential matrix with zeros on the diagonal (see Figure \ref{bdd} for example).

\begin{figure}[h]
 \centering
       \labellist
       \pinlabel $z$ at 30 150
       \pinlabel $1$ at 150 150
       \pinlabel $2$ at 150 30
       \pinlabel $3$ at 30 30
       \pinlabel $x_1$ at 120 -15
       \pinlabel $x_2$ at -15 90
       \pinlabel $x_3$ at 70 -15
       \pinlabel $x_1$ at 250 160
       \pinlabel $x_2$ at  340 90
       \pinlabel $x_3$ at 250 20
       \pinlabel $a(\rho_2)$ at 325 145
        \pinlabel $a(\rho_1)$ at 325 40
        \pinlabel $1$ at 240 90
              \pinlabel {$\begin{bmatrix} 
              0 & a(\rho_2) & 1\\
              0 & 0 & a(\rho_1)\\
              0 & 0 & 0\\
  \end{bmatrix} $} at 530 90
            \endlabellist
            \advance\leftskip-5cm
       \includegraphics[scale=.5]{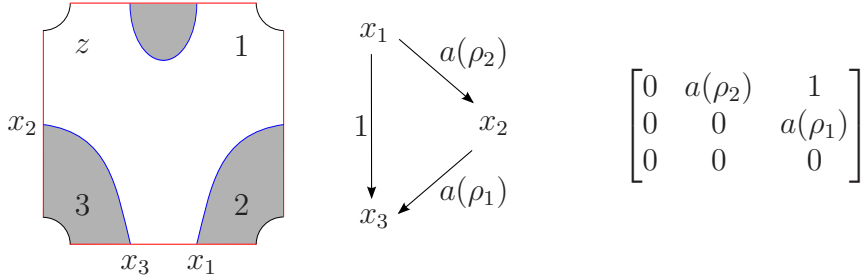}
       \vskip .2 cm
       \caption{An example of a bounded type $D$ structure. Left: A Heegaard diagram. The domains that contribute to $\delta$ are shaded. Center: The type $D$ structure for the diagram. Right: A matrix representation of the map $\delta$ with respect to the given ordering of the generating set. Note that the notation $\rho_i$ here comes from the labels on the diagram and carries a different meaning than in Section \ref{z2sec}.}
       \label{bdd}
\end{figure}

So chose an upper triangular differential matrix for $M$ and let $x$ be the generator corresponding to the bottommost row of the matrix. Then $x$ is a cycle, and, now thinking of $M$  as a left $\cala$-module, we have a distinguished triangle
$$\cala I_D(x)[m(x)]\to M\to M/\cala x ,$$
so $[M] = (-1)^{m(x)}[ \cala I_D(x)] + [M/\cala x]$. The matrix for $M/\cala x$ is obtained from the matrix for $M$ by removing the last row and column, hence $M/\cala x$ is bounded and can be generated by fewer elements than $M$, so by hypothesis $[M/\cala x]$ is the sum of symbols of elementary projectives. Thus, $[M]$ is of that form too. More precisely, applying the distinguished triangle process repeatedly by moving up the matrix,  we see that 
\begin{equation*}[M] = \sum_{x \in \mathfrak S(M)}(-1)^{m(x)} [\cala I_D(x)].
\end{equation*}

 So far we have shown that $K_0(\cala)$ is generated (over $\Z$) by the symbols of elementary projectives, , i.e. modules of the form  $\cala I(\sss)$. 

We now switch to the language of modules over $\cala$. 

The algebra $\cala=\az$ has a $\Z$-grading given by the total support of an element in $H_1(Z, \bf{a})$, i.e. the sum of the coefficients of the projection from $\gr'(a)\in G'(\zz)$ onto $H_1(Z, \bf{a})$. Since we only work with upward going and horizontal strands, this grading is in fact by non-negative numbers, and the degree $0$ part is precisely  the ground ring $\mathcal I$. The part of positive degree, call it $\cala_+$, is an augmentation ideal - dividing by it yields $\mathcal I$. This augmentation map $\cala\to \mathcal I$ induces a map of the categories of modules over $\cala$ and $\mathcal I$ by $M\mapsto M/\cala_+M$. More precisely, an $\cala$-module homomorphism $f:M\to N$ maps elements of $\cala_+M$ to $\cala_+N$, so it descends to an $\mathcal I$-module homomorphism $f':M/\cala+M\to N/\cala_+N$. Thus, distinguished triangles in $\mathcal P(\cala)$ are sent by the augmentation map to distinguished triangles in $\mathcal P(\mathcal I)$.  So a relation between the symbols of modules in $K_0(\cala)$ implies a relation between the corresponding symbols in $K_0(\mathcal I)$.

Note that there is a correspondence between primitive idempotents and elements of $\Lambda^\ast(H_1(F; \Z))$ given by 
$I(\sss)\to a_{\sss}$, so we can identify $K_0(\mathcal I) =K_0((\Z/2)^{2^{2k}})\cong\Z^{2^{2k}}$ with $\Lambda^\ast(H_1(F; \Z))$ via  $[\cala I(\sss)] = a_{\sss}$.
Thus,
\begin{equation*}
K_0(\cala)= \Lambda^\ast(H_1(F; \Z)).  
\end{equation*}

\end{proof}

\section{$K_0$ of Alexander-graded dg modules over $\az$}\label{alexsec}

In this section, we define a relative $\frac{1}{2}\Z$-grading $a$ on the dg algebra $\az$, as well as on the left type $D$ structures and right $\cala_{\infty}$-modules over $\az$. We refer to it as \emph{Alexander grading}, and show that it has properties similar to the Alexander grading on $\cfkhat$ for links in closed $3$-manifolds. For now we restrict to the case of torus boundary, and define an Alexander grading for $\cfdhat$ of knot complements and for $\cfahat$ of knots in the solid torus.

The pointed matched  circle $\zz = \zz(T^2)$ for a torus is depicted in in Figure \ref{fig:zc}. Four points are matched into two pairs $\alpha_1$ and $\alpha_2$, and  divide the circle into the four upward-oriented arcs $\rho_0$, $\rho_1$, $\rho_2$, and $\rho_3$ (not to be confused with the $\rho_i$ in Section \ref{z2sec}), where $\rho_0$ contains the basepoint $z$. The algebra $\az$ has two idempotents, one for each $\alpha_i$, and $6$ Reeb elements, coming from the Reeb chords $\rho_1, \rho_2$ and $\rho_3$ (also see  \cite[Section 11.1]{bfh2}). 

\psfrag{a1}{$\alpha_1$}
\psfrag{a2}{$\alpha_2$}
\psfrag{r0}{$\rho_0$}
\psfrag{r1}{$\rho_1$}
\psfrag{r2}{$\rho_2$}
\psfrag{r3}{$\rho_3$}
\psfrag{zz}{$z$}
\psfrag{z}{$^z$}

\begin{figure}[h]
\centering
 \includegraphics[scale=.5]{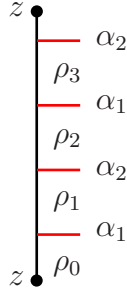}
  \caption{The circle $\zz$ in the case of torus boundary}
  \label{fig:zc}
\end{figure}

Recall that the unrefined grading on the torus algebra takes values in the group $G'$ which consists of quadruples $(j; r_1, r_2, r_3)$, with $j\in \frac{1}{2}\Z$, $r_i\in \Z$, and $j$ is an integer if $r_2$ and $r_1-r_3$ are even. There is also a refined grading by a subgroup $G$, different from $G(\zz)$, which consists of triples $(j; q_1, q_2)$ with $j, q_1, q_2 \in \frac{1}{2}\Z$ and $q_1+q_2\in \Z$. For the group law on $G'$ and $G$ and further details, see   \cite[Section 11.1]{bfh2}. 

We have a surjection 
\begin{align*}
G' & \twoheadrightarrow H_1(Z, \bf a)\\
(j; r_1, r_2, r_3) & \mapsto (r_1, r_2, r_3)
\end{align*}
As with the refined grading group $G(\zz)$, we can think of $G$ as a $\Z$ central extension of $H_1(T^2; \frac{1}{2}\Z)$. In other words, we have a map 
\begin{align*}
G & \to H_1(T^2; \Q)\\
(j; q_1, q_2) & \mapsto (q_1, q_2)
\end{align*}
with image inside $\frac{1}{2}\Z\times \frac{1}{2}\Z$ and kernel generated by $\lambda = (1; 0, 0)$.

Let $K$ be a knot in $S^3$ and let $\cfdhat(K, n)$ be a type $D$ structure for the $n$-framed complement $S^3\setminus K$. Fix a generator $x$ and recall that we can choose the sign of the generator $B$ of $\pi_2(x, x)$ so that $B$ does not cover the basepoint $z$, and has multiplicities $1, n+1, n$ at  the regions corresponding to $\rho_1, \rho_2, \rho_3$, respectively.

We can define the Alexander grading in two ways: by composing $\gr'$ with map from $G'$ to $\frac{1}{2}\Z$, or by composing $\gr$ with a map from the subgroup $G$ to $\frac{1}{2}\Z$. The two maps to $\frac{1}{2}\Z$ are obtained by composing the two maps above with 
\begin{align*}
H_1(Z, \bf a) & \to \textstyle \frac{1}{2}\Z\\
 (r_1, r_2, r_3) & \mapsto  \textstyle{\frac{n+1}{2}r_1 + \frac{n-1}{2}r_2 + \frac{-n-1}{2}r_3}
\end{align*}
and
\begin{align*}
H_1(T^2; \Q) & \to  \textstyle\frac{1}{2}\Z\\
 (q_1, q_2) & \mapsto nq_1-q_2. 
\end{align*}
The gradings $g'(B)$ and $g(B)$ are in the kernel of the respective maps, so we get maps from the quotients $G'/P(x)$ and $G/P(x)$ to $\frac{1}{2}\Z$, i.e. an \emph{Alexander grading} on $\cfdhat(K, n)$ by $\frac 12\Z$, which we denote by $a$. 

 Also note that the two maps commute with the grading maps (i.e. transitioning between $G$ and $G'$), hence the two maps define the same grading on $\cfdhat(K, n)$ by $\frac 12\Z$. Note that this grading agrees with the function $S$ from  \cite[Equation 11.39]{bfh2}.

Next, we define the Alexander grading for a knot in the solid torus. One might like to just add the number of times we pass through the second basepoint of the Heegaard diagram to the Alexander grading above, but we also have to keep track of the homological class of the knot in the solid torus.

Given a $0$-framed solid torus and a knot $K$ in it with homology class $[p]$, fix a generator $x$ of $\cfahat(K)$, and note that we can choose the generator $B$ of $\pi_2(x,x)$ to have multiplicities $ 0, 1, 1$ at $\rho_1, \rho_2, \rho_3$ respectively, and to avoid the basepoint  $z$. This $B$ covers $w$ with multiplicity  $p$: the boundary of $B$ is a set of complete $\alpha$ and $\beta$ circles, and the complete arc $\alpha_1$. By capping off  all circles with the disks they bound, we get an immersed surface in the solid torus whose boundary is the meridian of the solid torus. In other words, the surface we obtain from $B$ is homologically  equivalent to  the disk $D^2$ (with positive orientation) that the meridian bounds, hence intersects the knot homologically $p$ any times. Hence, in the Heegaard diagram, $B$ covers the basepoint $w$ a total of $p$ many times. 
 
This time we have the algebra graded by $G'\times \{0\}$ or $G\times \{0\}$ as subgroups of $G'\times \Z$ or $G'\times \Z$, and corresponding maps to $H_1(Z, \bf a)\times \Z$ or $H_1(T^2; \Q)\times \Z$ defined as above on the first factor, and as the identity on the second.  Define the Alexander grading $a$ on the algebra as the composition of these maps with the maps to $\frac{1}{2}\Z$ given by $(r_1, r_2, r_3; d)\mapsto d - p(r_2+r_3)$ and $(q_1, q_2; d)\mapsto d- pq_2$ respectively.

The domain $B$ above has grading of the form $g'(B)= (\_; 0, 1, 1; p)$ in $G'$, or $g(B) = (\_; 0, 1; p)$ in $G$, and the grading on $\cfahat(K)$ takes values in $G'\times \Z/\langle g'(B)\rangle$ or $G\times \Z/\langle g(B)\rangle$, respectively. The gradings of $B$ are in the kernels of the maps, so we get well defined maps from the quotients $G'\times \Z/\langle g'(B)\rangle$ and  $G\times \Z/\langle g(B)\rangle$ to  $\frac{1}{2}\Z$, and also note that the two maps commute with the grading maps,  hence we have  a well defined grading on $\cfahat (K)$, which we also call the \emph{Alexander grading} and denote by $a$. 

Note that a $\frac{1}{2}\Z$-grading introduces relation (3) in Definition \ref{k0def}, and 
 the Grothendieck group for left type $D$ structures and right $\ainf$-modules over $\az$ becomes
$$K_0(\az_{gr}) \cong \Lambda^\ast(H_1(F; \Z))\otimes \Z[t^{1/2}, t^{-1/2}].$$

\section{Tensor products}\label{tensorsec}
We first show that as a relative grading, the grading $m$ defined in Section \ref{z2sec} agrees with the relative Maslov grading for closed manifolds. 

Let $Y_1$ and $Y_2$ be bordered $3$-manifolds which agree along their boundary, with Heegaard diagrams $\mathcal H_1$ and $\mathcal H_2$ which can be glued along their boundary $\zz = \bdy \mathcal H_1 = -\bdy \mathcal H_2$ to form a closed Heegaard diagram $\mathcal H = \mathcal H_1\cup_\bdy \mathcal H_2$ representing the closed $3$-manifold $Y = Y_1\cup_\bdy Y_2$.  Let $\mathfrak s\in \mathrm{spin}^c(Y)$, and let $\mathfrak s_i = \mathfrak s|_{Y_i}$. 

Let $u, v\in \cfahat(\mathcal H_1, \mathfrak s_1)$ and  $x, y\in \cfdhat(\mathcal H_2, \mathfrak s_2)$ be homogeneous elements with respect to the grading $\gr$, and suppose $u$ and $x$ occupy complementary $\alpha$-arcs, and so do $v$ and $y$, and $u\boxtimes x$ and $v\boxtimes y$ are in the same $\textrm{spin}^c$ structure $\mathfrak s$.

\begin{proposition}\label{maslov}
Let $t$ be the relative Maslov grading mod $2$ between $u\boxtimes x$ and $v\boxtimes y\in \cfhat (\mathcal H, \mathfrak s)$ (see \cite[Theorem 10.42]{bfh2}). Then
$$[m(u) + m(x)]-[m(v)+m(y)] = t \text{\emph{ mod} } 2$$ 
\end{proposition}

\begin{proof}
Recall that $\cfahat(\mathcal H_1, \mathfrak s_1)$ is graded by $P_1(x_1)\backslash G(\zz)$, where $x_1\in \mathfrak S(\mathcal H_1, \mathfrak s_1)$ is a chosen base generator, and $P_1(x_1)$ is the image of $\pi_2(x_1, x_1)$ in $G(\zz)$.  Similarly, $\cfdhat(\mathcal H_2, \mathfrak s_2)$ is graded by $G(\zz)/R(P_2(x_2))$, where $x_2\in \mathfrak S(\mathcal H_2, \mathfrak s_2)$ is a chosen base generator and $P_2(x_2)$ is the image of $\pi_2(x_2, x_2)$ in $G(\zz)$.  Recall that the map $f$ from Section \ref{z2sec} maps any element of $P_1(x_1)$ or $R(P_2(x_2))$ to zero. 

By \cite[Theorem 10.42]{bfh2}, 
$$\gr^{\boxtimes}(u\boxtimes x) = \lambda^t \gr^{\boxtimes}(v\boxtimes y)\in P_1(x_1)\backslash G(\zz)/R(P_2(x_2)).$$
This means that if we fix representatives $g_u, g_v, g_x, g_y\in G(\zz)$ so that $[g_u] = \gr(u)\in P_1(x_1)\backslash G(\zz)$, etc., then 
$$[g_u g_x] = [\lambda^t g_v g_y]\in P_1(x_1)\backslash G(\zz)/R(P_2(x_2)).$$
Then there exist $h_1\in P_1(x_1)$ and $h_2\in R(P_2(x_2))$, such that 
$$h_1 g_u g_x h_2 = \lambda^t g_v g_y\in G(\zz).$$
Applying $f$ to both sides, we see that
$$f(h_1 g_u g_x h_2) = f(\lambda^t g_v g_y),$$
so 
\begin{align*}
0 &= f(h_1 g_u g_x h_2) - f(\lambda^t g_v g_y)\\
 &= f(h_1)+f(g_u)+ f(g_x) + f(h_2) - tf(\lambda) - f(g_v) -  f(g_y)\\
 &=  f(g_u)+ f(g_x) - t - f(g_v) -  f(g_y)\\
 &= \bar f (\gr(u)) +\bar f(\gr(x)) - t -\bar f(\gr(v)) - \bar f(\gr(y))\\
 &= m(u) + m(x) - t - m(v) - m(y).\qedhere
\end{align*}
\end{proof}

A similar statement can be made about the Alexander grading on the manifolds studied in Section \ref{alexsec}.
Let $\mathcal H_1$ be a Heegaard diagram for a knot in a $0$-framed solid torus $(S^1\times D^2, C)$, and let $\mathcal H_2$ be a Heegaard diagram for a $0$-framed knot complement $(S^3\setminus K, 0)$, so that $\mathcal H = \mathcal H_1\cup_\bdy \mathcal H_2$ is a Heegaard diagram for $(S^3, K_C)$. Recall that there is only one $\textrm{spin}^c$ structure for each of the bordered manifolds and for $S^3$. Also recall that for knots, we enhance the grading group $G$ to  $G\times \Z$, and in this case $\cfahat(\mathcal H_1)$ is graded by the coset space $\ha\backslash G$, and $\cfdhat(\mathcal H_2)$ is graded by a coset space $G/\hd$, where $\ha$ and $\hd$ are the images in $G$ of the the groups of periodic domains of $\HH_1$ and $\HH_2$. Every element of the double-coset space  $\ha\backslash G/\hd$ has a representative of the form $(j; 0, 0; d)$, where $j,d\in \Z$, so we can think of $\cfahat(\HH_1)\boxtimes\cfdhat(\HH_2)$ as graded by $\Z\times \Z$.  The second factor, i.e. the number $d$, is in fact the relative Alexander grading for knot Floer homology. For more detail on these facts, we refer the reader to the the discussion preceding  \cite[Theorem 11.21]{bfh2}, and the computation example in \cite[Section 11.9]{bfh2}.

\begin{proposition}\label{alex_tensor}
Let $u, v\in \cfahat(\mathcal H_1)$ and  $x, y\in \cfdhat(\mathcal H_2)$ be homogeneous elements with respect to the grading $\gr$, and suppose $u$ and $x$ occupy complementary $\alpha$-arcs, and so do $v$ and $y$.
Let $A$ be the relative Alexander grading between $u\boxtimes x$ and $v\boxtimes y\in \cfkhat (\mathcal H)$. Then
$$(a(u) + p\cdot a(x))-(a(v)+p\cdot a(y)) = A $$ 
\end{proposition}

\begin{proof}
The Maslov component does not affect the computation, so we leave it blank. Recall that in this $0$-framed case we can assume that $h_A$ has form $(\_; 0, 1; p)$ and $h_D$ has form $(\_; 1, 0; 0)$.

Fix representatives 
\begin{align*}
g_u &= (\_; p_u, q_u; d_u),\\
 g_v &= (\_; p_v, q_v; d_v), \\
 g_x &= (\_; p_x, q_x; d_x), \\
 g_y &= (\_; p_y, q_y; d_y)
\end{align*}
in $G$ so that $[g_u] = \gr(u)\in \ha\backslash G$, etc. Then
$$[g_ug_x]= [h_A^{-q_u-q_x}g_ug_x h_D^{-p_u-p_x}]= [(\_; 0, 0; d_u - pq_u-pq_x)],$$
$$[g_vg_y]= [h_A^{-q_v-q_y}g_vg_y h_D^{-p_v-p_y}]= [(\_; 0, 0; d_v - pq_v-pq_y)],$$
and so $A = A(u\boxtimes x)-A(v\boxtimes y) = (d_u - pq_u-pq_x)-(d_v - pq_v-pq_y)$, which equals precisely
$(a(u)+ p\cdot a(x))-(a(v) + p\cdot a(y))$.
\end{proof}

Now we define a product operation on $\Lambda^\ast H_1(F; \Z)$. Define an inner product on the vector space $H_1(F; \Z)$ by 

\begin{displaymath}
\langle a_i, a_j\rangle = \left\{ \begin{array}{ll}
1 \quad & \textrm{if  $i=j$}\\
0 & \textrm{otherwise,}
\end{array} \right.
\end{displaymath}
i.e. so that  $\{a_1, \ldots, a_{2k}\}$ is an orthonormal basis. This extends to an  inner product on  $\Lambda^n H_1(F; \Z)$ by 
$$\langle v_1\wedge\cdots\wedge v_n, w_1\wedge\cdots\wedge w_n \rangle = \det (\langle v_i, w_j\rangle)$$
and $\{a_J| |J|=n\}$ is an orthonormal basis for $\Lambda^n H_1(F; \Z)$. We define an operation $\cdot$ on $\Lambda^\ast  H_1(F; \Z)$ by 
$$x\cdot y := \langle x_1, y_1\rangle +  \langle x_2, y_2\rangle+ \cdots+  \langle x_{2k}, y_{2k}\rangle,$$
where $x_i$ and $y_i$ are the $\Lambda^i$ components of $x$ and $y$, respectively.

\begin{proof}[Proof of Theorem \ref{intro_tensor}]
For the first part,  $M\boxtimes N$ inherits a $\Z/2$ grading from $M$ and $N$ by $m(x\boxtimes y) = m(x)+ m(y)$. We have
\begin{align*}
[M]\cdot[N] &= \sum_{x\in \mathfrak S(M)}(-1)^{m(x)}h(x)\cdot \sum_{y\in \mathfrak S(N)}(-1)^{m(y)}h(y)\\
&=\left(\sum_{\sss\subset[2k]} a_{\sss}\sum_{x\in \mathfrak S(M), h(x) = a_{\sss}}(-1)^{m(x)}\right)   \cdot     \left(\sum_{\ttt\subset[2k]} a_{\ttt}\sum_{y\in \mathfrak S(N), h(y) = a_{\ttt}}(-1)^{m(y)}\right)\\
&=\sum_{\substack{\sss, \ttt\in [2k]\\|\sss| = |\ttt|}}\langle a_{\sss}, a_{\ttt}\rangle \sum_{\substack{x\in \mathfrak S(M), h(x) = a_{\sss}\\ y\in \mathfrak S(N), h(y) = a_{\ttt}}} (-1)^{m(x)} (-1)^{m(y)}\\
&= \sum_{\substack{x\in \mathfrak S(M),  y\in \mathfrak S(N)\\ h(x) = a_{\sss} = h(y)}} (-1)^{m(x)+ m(y)}\\
&= \sum_{x\boxtimes y} (-1)^{m(x)+ m(y)}\\
&= \chi(M\boxtimes N).
\end{align*}
 The second part is the specialization of the first part to  $\cfahat$ and $\cfdhat$ with the grading $m$ defined in Section \ref{z2sec} and $\cfhat$ graded by the Maslov grading. We remark  that $[\cfdhat]$ and $[\cfahat]$ are elements of $\Lambda^k H_1(F; \Z)$. The last step in the above equation follows, up to an overall sign for each $\mathrm{spin}^c$ structure, from Proposition \ref{maslov}. 
\end{proof}

\begin{proof}[Proof of Theorem \ref{intro_alex}]
This is the Alexander-graded version of Theorem \ref{intro_tensor}. Recall that homology solid tori have a unique $\mathrm{spin}^c$ structure. For $\mathcal H_1$ and $\mathcal H_2$ as in Proposition \ref{alex_tensor},
\begin{align*}
[\cfahat(\mathcal H_1)] &= \sum_{\substack{x\in \mathfrak S(\mathcal H_1)\\I_A(x) = \iota_0}}(-1)^{m(x)}t^{a(x)}a_1 + \sum_{\substack{x\in \mathfrak S(\mathcal H_1)\\I_A(x) = \iota_1}}(-1)^{m(x)}t^{a(x)}a_2\\
[\cfdhat(\mathcal H_2)]  &= \sum_{\substack{y\in \mathfrak S(\mathcal H_2)\\I_D(y) = \iota_0}}(-1)^{m(y)}t^{a(y)}a_1 + \sum_{\substack{y\in \mathfrak S(\mathcal H_2)\\I_D(y) = \iota_1}}(-1)^{m(y)}t^{a(y)}a_2.
\end{align*}
We multiply out and see that $[\cfahat(\mathcal H_1)]\cdot[\cfdhat(\mathcal H_2)] $ equals
\begin{align*}
 \sum_{\substack{x\in \mathfrak S(\mathcal H_1), y\in \mathfrak S(\mathcal H_2)\\I_A(x) = \iota_0=I_D(y)}} & (-1)^{m(x)+ m(y)}t^{a(x)+ p a(y)} 
+ \sum_{\substack{x\in \mathfrak S(\mathcal H_1), y\in \mathfrak S(\mathcal H_2)\\I_A(x) = \iota_1=I_D(y)}}(-1)^{m(x)+ m(y)}t^{a(x)+ p a(y)} \\
&=\sum_{\substack{x\in \mathfrak S(\mathcal H_1), y\in \mathfrak S(\mathcal H_2)\\x\boxtimes y\neq 0}}(-1)^{m(x)+ m(y)}t^{a(x)+ p a(y)} \\
&= \sum_{\substack{x\in \mathfrak S(\mathcal H_1), y\in \mathfrak S(\mathcal H_2)\\x\boxtimes y\neq 0}}(-1)^{M(x\boxtimes y)}t^{A(x\boxtimes y)} \\
&= \pm t^k\chi(\cfkhat(K_C))
\end{align*}

Gluing a Heegaard diagram for the unknot in the $0$-framed solid torus with one generator (see \cite[Section 3]{cables}) to a diagram for $(S^3\setminus K, 0)$ shows that the $a_1$ component of  $[\cfdhat(S^3\setminus K, 0)]$ is  $\chi(\cfkhat(K)) =  \Delta_K(t)$. 
Similarly, gluing a diagram with one generator for the $0$-framed complement of the  unknot to a diagram for $C\hookrightarrow S^1\times D^2$ in a $0$-framed $S^1\times D^2$ shows that the $a_1$ component of $[\cfahat(S^1\times D^2, C)]$ is $\chi(\cfkhat(C)) =  \Delta_C(t).$

Proving that the $a_2$ component of  $[\cfdhat(S^3\setminus K, 0)]$ vanishes requires a different argument. Say $\rank \hfkhat(S^3, K) = n$ and  assume we have a basis  $\{x_0, \ldots, x_{2n}\}$ for $\cfkm(K)$ which is both horizontally and vertically simplified.  We illustrate $\cfkm(K)$ on the $(U, A)$ lattice as in \cite[Section 3]{cables}. Recall this means we have $2n+1$ points $\{ \xi_0, \ldots, \xi_{2n}\}$ representing the basis, and a vertical  (respectively horizontal) arrow from $x_i$ to $x_j$ of length $l$ is represented by a vertical (respectively horizontal) arrow of length $l$ pointing down (respectively to the left) staring at $\xi_i$ and ending at $\xi_j$. Note that there is at most one vertical and at most one horizontal arrow starting or ending at each basis element, and there is a unique element $\xi^v$ with no in-coming or out-going vertical arrow, and a unique element $\xi^h$ with no in-coming or out-going horizontal arrow. See Figure \ref{cfdgraph}.

Recall that we can obtain $\cfdhat(S^3\setminus K,0)$ by replacing arrows with chains of coefficient maps, and adding one more chain from $\xi^v$ to $\xi^h$, called the unstable chain.

Choose $\iota_0$ to be the base idempotent, and define $\psi(\iota_1) = (\frac{1}{2}; 1, 0, 0)\in G'(4)$. This refinement give rise to a $G(\zz)$ grading, and hence to a $\Z/2$ grading $m$. With this choice, we list the relative $(m,a)$ gradings of a horizontal chain of length $l$, a vertical chain of length $l$, and the unstable chain when $\tau(K)=0$, $\tau(K)>0$, and $\tau(K)<0$, in this order.

\begin{displaymath}
\xymatrix{ 
(0,0)\ar[r]^{D_1}&(0,-\frac{1}{2}) &  (0,-\frac{3}{2}) \ar[l]_{D_{23}}& \dots  \ar[l]_{D_{23}} &   (0,\frac{1}{2}-l)  \ar[l]_{D_{23}}& (1, -l)\ar[l]_{D_{123}}\\
(0,0)\ar[r]^{D_3}& \ar[r]^{D_{23}} (0,\frac{1}{2}) &  (0,\frac{3}{2}) \ar[r]^{D_{23}}& \dots  \ar[r]^{D_{23}} &   (0,l-\frac{1}{2}) \ar[r]^{D_2}& (1, l) \\
(0,0)\ar[r]^{D_{12}} & (0,0) & & & & \hspace{68pt}\\
(0,0)\ar[r]^{D_1}&(0,-\frac{1}{2})  &  (0,-\frac{3}{2}) \ar[l]_{D_{23}}& \dots\ar[l]_{D_{23}} & (0,\frac{1}{2}-2\tau)  \ar[l]_{D_{23}\phantom{123}}&\ar[l]_{D_3}(0, -2\tau)\\
(0,0)\ar[r]^{D_{123}}&(1,\frac{1}{2}) \ar[r]^{D_{23}} &  (1,\frac{3}{2}) \ar[r]^{D_{23}}& \dots\ar[r]^{D_{23}\phantom{1234}}&(1,-2\tau-\frac{1}{2})\ar[r]^{D_2}& (0,-2\tau)
}
\end{displaymath}
\\

We plot the chains on the $(U, A)$ coordinate system so that the grading $a$ of a generator with coordinates $(x,y)$ is given by $x-y$. Draw each chain corresponding to a vertical or horizontal arrow  also as a vertical or horizontal chain, respectively, and represent coefficient maps between a generator in $\iota_0$ and a generator in $\iota_1$ by arrows of length $\frac{1}{2}$, and coefficient maps between two generators in $\iota_1$ by arrows of length one. The choice for the unstable chain depends on $\tau(K)$.

{\bf{\bigskip}{\noindent}{\underline{Case 1:}}} $\tau(K) = 0$.
Ignore the $D_{12}$ map, and  identify $\xi^v$ and $\xi^h$ if they are not the same basis element. Note this may not result in the correct model for $\cfdhat(S^3\setminus K, 0)$, but  the information about the $\iota_1$ elements is intact, which is all we are interested in. 

{\bf{\bigskip}{\noindent}{\underline{Case 2:}}} $\tau(K) > 0$.

Note that in this case $\xi^v$ is $\tau(K)$ units above and $\tau(K)$ units to the left of $\xi^h$. Draw the unstable chain in an $L$-shape, as follows. Starting at $\xi^v$,  represent the first coefficient map by a vertical arrow of length $\frac{1}{2}$, so that the first $\iota_1$ element is half a unit  below $\xi^v$. Proceed downwards, drawing the $D_{12}$ maps to have length one, until half the $\iota_1$ elements have been plotted. Repeat the process for the other half, starting at $\xi^h$ and going to the left. Connect the middle two elements by a straight arrow to represent the coefficient map between them. 

{\bf{\bigskip}{\noindent}{\underline{Case 3:}}} $\tau(K) < 0$.

Here $\xi^v$ is $|\tau(K)|$ units below and $|\tau(K)|$ units to the right of $\xi^h$. Rotate the construction for Case 2 by $180^{\circ}$.
 
Figure \ref{cfdgraph} illustrates the above description with a couple of examples.

Note that each element lies on a line $L_t$ of slope $1$ passing through $(t, 0)$, for some $t\in \frac{1}{2}\Z$. For $\iota_0$ elements $t\in \Z$, and for $\iota_1$ elements $t\in \Z+\frac{1}{2}$. If $a(x)-a(y) = t \in \frac{1}{2}\Z$, then the line through $x$ is $t$ units above the line through $y$. 

Let $G$ be the graph on vertices $\xi_0, \ldots, \xi_{2n}$ and edges the $2n+1$ chains (note that if $\tau(K)=0$ we may have only $2n$ vertices and $2n$ edges), embedded as above. Every vertex has degree $2$, so $G$ is a union of cycles. Endow edges with the orientation induced by the horizontal and vertical arrows for $\cfkm$, and orient the edge corresponding to the unstable chain to start at $\xi^v$ and end at $\xi^h$.  

Rotate the plane clockwise by $45^{\circ}$, so that the lines $L_t$ are now horizontal, and smoothen  $G$ locally at the vertices, so we can think of it as an immersion of a union of circles.  The vertices now comprise the local minima, local  maxima, and the points with vertical tangents of the immersion. A vertex $v\neq \xi^v, \xi^h$ is a local maximum if it has an incoming horizontal edge and an outgoing vertical edge, and a local minimum if it has an incoming vertical edge and an outgoing horizontal edge. The vertex $\xi^v$ is a local maximum if it has  an incoming horizontal edge and $\tau(K)\geq0$, and a local minimum if it has an outgoing horizontal edge and $\tau(K)\leq 0$. Otherwise it has a vertical tangent.  Similarly, $\xi^h$ is a local maximum if it has  an outgoing vertical edge and $\tau(K)\leq 0$, and a local minimum if it has an incoming vertical  edge and $\tau(K)\geq 0$. Note that this covers the case when $\xi^v=\xi^h$.

\begin{figure}[pb!]
 \centering
       \labellist
       \pinlabel $\xi^v$ at 12 575
           \pinlabel $\xi^h$ at 170 414
       \pinlabel $\xi^h$ at 330 567
           \pinlabel $\xi^v$ at 430 440
            \endlabellist
       \includegraphics[scale=.8]{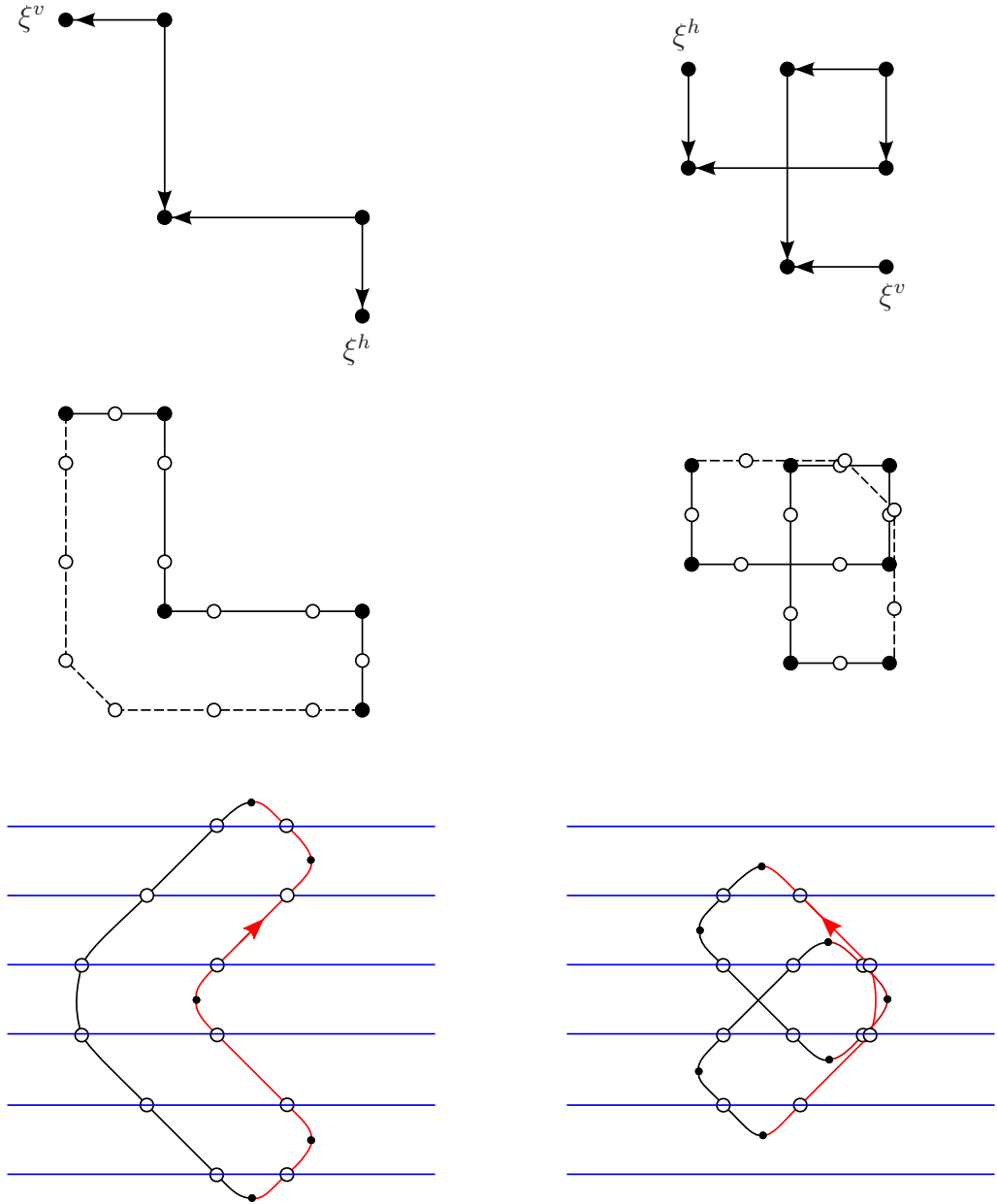} 
       \vskip .2 cm
       \caption{Two examples of the graphical interpretation of $\cfkm$ and $\cfdhat$. Left: The  $(3,4)$ torus knot ($\tau = 3$). Right: the $(2,-1)$ cable of the left-handed trefoil ($\tau = -2$). On top are the models for $\cfkm$, below are the models for $\cfdhat$, with dashed unstable chain, and on the bottom are the rotated, smoothened graphs. Red and black represent opposite $m$ gradings.}
       \label{cfdgraph}
\end{figure}

Tracing the $\iota_1$ elements along a connected component of $G$, observe that the $m$ grading changes exactly when passing through a local maximum or minimum. On the other hand, think of a connected component as an immersed circle, ignore the orientations of the edges, and fix an orientation for the circle. The derivative of the height function changes sign exactly at the local maxima or minima, so two $\iota_1$ elements have the same $m$ grading exactly when the derivatives of the height function at their coordinates have the same sign. For any $t\in \Z+\frac{1}{2}$, the line $L_t$ crosses $G$ away from any local minima and maxima, and the oriented intersection number of $L_t$ and $G$ is zero, since $G$ consists of immersed circles. This means that half the intersection points have positive derivative, and the other half have negative derivative (see Figure \ref{cfdgraph}). In other words, half of the $\iota_1$ elements of a given $a$ grading have $m$ grading 0, and the other half have $m$ grading $1$, and so they cancel each other out in the summation for $[\cfdhat(S^3, K), 0]$, i.e.   
$$\sum_{x\in \mathfrak S(\cfdhat(S^3\setminus K, 0)), I_D(x)=\iota_1} (-1)^{m(x)}t^{a(x)}= 0.$$
In other words, the $a_2$ component of $[\cfdhat(S^3\setminus K, 0)]$ is zero.

Last, we show that we do not in fact need the assumption that we have a basis which is both vertically and horizontally simplified. 

 Let  $\Xi = \{\xi_0, \ldots, \xi_{2n}\}$ be the plot of a vertically simplified basis with $\xi^v$ in position $(0,\tau)$, and let $H = \{\eta_0, \ldots, \eta_{2n}\}$ be the plot of a horizontally simplified basis with $\eta^h$ in position $(\tau, 0)$. Also plot the vertical and horizontal chains.

The symmetries of $\cfk^{\infty}$ discussed in \cite[Section 3.5]{hfk} imply that when we have a reduced complex, the vertical and horizontal complexes are isomorphic as graded, filtered complexes. The change of basis  for this isomorphism may not be a bijection, and in fact may not even map generators to homogeneous linear combinations of generators, but since the isomorphism preserves gradings and filtrations, we may deduce that the number of elements of $\Xi$ in position $(x,y)$ with given Maslov grading is the same as the number of elements of $H$ with the same grading in position $(y,x)$.

In addition, the symmetry  
$$\hfkhat_i(K,j)\cong \hfkhat_{i-2j}(K, -j)$$
implies that  there is the same number of elements of $\Xi$ with given parity of the Maslov grading in position $(x,y)$, as in position $(y,x)$. 

Since the grading $m$ agrees with the Maslov grading, by combining the two symmetries, we see that for a given $m$ grading ($0$ or $1$) there the number of elements of $\Xi$ of that grading in a given position $(x,y)$ equals the number of elements $H$ of the same grading in the same position.
In other words,  we can find a bijection $b:H\to \Xi$ that preserves coordinates and also preserves the $m$ grading. Identify the two bases under this bijection, i.e. think of  a horizontal  chain from $\eta_i$ to $\eta_j$, as a horizontal chain from $b(\eta_i)$ to $b(\eta_j)$, and think of the unstable chain as going from $\xi^v$ to $b(\eta^h)$. While the result of this identification may not represent $\cfdhat(S^3\setminus K, 0)$, it has the same graphical structure that we already analyzed in the case of a basis which is simultaneously horizontally and vertically simplified. The bigradings on the chains when moving along a connected component of the graph obey the same rules as before, since $b$ respects the $m$ grading, and by \cite[Lemma 3.2.5]{cables} the $a$ grading of any element is specified by its coordinates.  This allows is to make the same cancelation argument as before.
\end{proof}

\section{A $\Z/2$ grading via intersection signs}\label{z2int}

In Section \ref{z2sec} we defined a $\Z/2$  differential grading $m$  on the surface algebra $\az$ and the left or right modules over it, and showed that it agrees with the Maslov grading after tensoring. For the algebra, this grading was defined as a composition of the $G(\zz)$ grading  from  \cite{bfh2} with a homomorphism from $G(\zz)$ to $\Z/2$;  for the modules, it was defined as a composition of the $G(\zz)$-set grading with a quotient of the homomorphism from $G(\zz)$ to $\Z/2$.

Inspired by a similar definition in \cite{dsfh}, in this section we provide a more hands-on definition of the grading $m$, via intersection signs of $\alpha$- and $\beta$-curves on a Heegaard diagram. 

Let $\zz$ be a pointed matched circle, and let $k$ be the genus of the surface $F(\zz)$. Given a Heegaard diagram $\HH$ with $\bdy \HH=\zz$, recall that the $4k$ points  $\balpha\cap \zz$ come with an ordering $\lessdot$, induced by the orientation of $\zz\setminus z$ \cite[Section 3.2]{bfh2}. For any $\alpha$-arc $\alpha_i$, label its endpoints as $\alpha_i^-$ and $\alpha_i^+$, so that $\alpha_i^-\lessdot\alpha_i^+$, and 
order the $\alpha$-arcs so that $\alpha_1^-\lessdot\alpha_2^-\lessdot\ldots\lessdot\alpha_{2k}^-$. Write the matching  as $M(\alpha_i^-)= M(\alpha_i^+) = i$, i.e. so that an idempotent $I(\sss)$ corresponds to the set of $\alpha$-arcs indexed by $\sss\subset [2k]$. We recall  the definition of the function $J$ from Section \ref{k0sec}. Given a set $\sss \subset [2k]$, $J(\sss)$ is the multi-index, i.e. ordered set, $(j_1, \ldots, j_n)$ for which $ j_1<\ldots < j_n$ and $\{j_1, \ldots, j_n\} = \sss$.

We define a grading on the algebra $\az$ by  looking at the diagram for the bimodule $\az$ that  was studied in \cite{auroux} (labeled $(\hat F, \{\tilde \alpha^-_i\}, \{\tilde \alpha^+_i\})$), and in \cite[Section 4]{hfmor} (labeled $\mathrm{AZ}(\zz)$).  Figure \ref{az} is an example of $\mathrm{AZ}(\zz)$ when $\zz$ is the split pointed matched circle of genus $2$. Let  $\bdy_{\alpha}\textrm{AZ}(\zz)$ denote the boundary component of $\mathrm{AZ}(\zz)$ which intersects the $\alpha$-arcs, and Let  $\bdy_{\beta}\textrm{AZ}(\zz)$ denote the boundary component of $\mathrm{AZ}(\zz)$ which intersects the $\beta$-arcs. Order the $\alpha$-arcs and label their endpoints as above, i.e. following the orientation of $\bdy_{\alpha}\textrm{AZ}(\zz)\setminus z$, and do the same for the $\beta$-arcs,  i.e. following the orientation of $\bdy_{\beta}\textrm{AZ}(\zz)\setminus z$. For each $i$, orient $\alpha_i$ from $\alpha_i^-$ to $\alpha_i^+$, and  $\beta_i$ from $\beta_i^-$ to $\beta_i^+$. For any point $x\in \balpha\cap \bbeta$, define $s(x)$ to be the intersection sign of $\balpha$ and $\bbeta$ at $x$.  Note that the intersection sign of $\alpha_i$ and $\beta_i$ at the diagonal of the triangle is positive. 

\begin{figure}[tb!]
 \centering
       \labellist
       \pinlabel $z$ at 138 138
       \pinlabel $\alpha_1^-$ at 155 18
       \pinlabel $\alpha_2^-$ at 155 34
       \pinlabel $\alpha_1^+$ at 155 50
       \pinlabel $\alpha_2^+$ at 155 66
       \pinlabel $\alpha_3^-$ at 155 82
       \pinlabel $\alpha_4^-$ at 155 98
       \pinlabel $\alpha_3^+$ at 155 114
       \pinlabel $\alpha_4^+$ at 155 130
       \endlabellist
       \includegraphics[scale=.95]{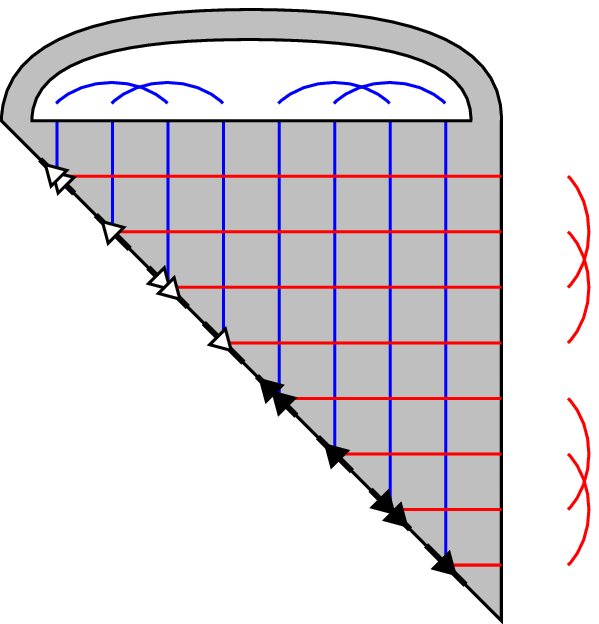} 
       \vskip .2 cm
       \caption{The diagram $\mathrm{AZ}(\zz)$.}
       \label{az}
\end{figure}

Recall that the generators $\mathfrak S(\mathrm{AZ}(\zz))$ are in one-to-one correspondence with the standard generators of $\az$ by strand diagrams. We will denote a generator of $\az$ and the corresponding generator in $\mathfrak S(\mathrm{AZ}(\zz))$ the same way.  Given a generator ${\bf a}$ of $\az$, write its representative in  $\mathfrak S(\mathrm{AZ}(\zz))$ as an ordered subset $\aaa = (x_1, \ldots, x_p)$ of $\balpha\cap\bbeta$, with $x_i$ ordered according to the occupied $\alpha$-arcs.  For a generator $\xxx\in \mathfrak S(\mathrm{AZ}(\zz))$, let  $o_\alpha(\xxx)$ be the set of occupied $\alpha$-arcs, and let  $o_\beta(\xxx)$ be the set of occupied $\beta$-arcs.
Define $\sigma_{\xxx}$ to be the permutation for which 
\begin{align*}
x_1 &\in \alpha_{i_1}\cap \beta_{j_{\sigma_{\xxx}(1)}}\\
& \hspace{7pt} \vdots\\
x_p &\in \alpha_{i_p}\cap \beta_{j_{\sigma_{\xxx}(p)}}\\
\end{align*}
where $(i_1, \ldots, i_p) = J(o_\alpha(\xxx))$ and $(j_1, \ldots, j_p) = J(o_\beta(\xxx))$. In other words, $\sigma_{\xxx}$ is the permutation arising from the induced orders on the two sets of occupied arcs.
Define the sign of $\xxx$ by 
$$s(\xxx) = \textrm{sign}(\sigma_{\xxx})\prod_{i=1}^p s(x_i).$$

\begin{lemma}
The sign assignment $s$ induces a differential grading $m$ on $\az$, in the sense that the unique function $m: \mathfrak S(\mathrm{AZ}(\zz))\to \Z/2$ for which $s= (-1)^m$ is a differential grading.
\end{lemma}

\begin{proof}
If a generator $\yyy$ is in the differential of $\xxx$, then $\xxx$ and $\yyy$ are equal size, say $p$, as subsets of $\balpha\cap \bbeta$. There is a rectangle connecting $\xxx$ to $\yyy$, so $\xxx$ and $\yyy$ differ exactly at the vertices of the rectangle, say $x_i\neq y_i$ and $x_j\neq y_j$. Then $\prod_{i=1}^p s(x_i) \cdot \prod_{i=1}^p s(y_i) = s(x_i)s(x_j)s(y_i)s(y_j)=1$, and $\textrm{sign}(\sigma_{\xxx}) = - \textrm{sign}(\sigma_{\yyy})$, hence $s(\yyy) = -s(\xxx)$.

For the $m_2$ multiplication, note that if $\xxx$ and $\yyy$ are generators with $\xxx\cdot \yyy = \zzz$, we can see this as a set of half-strips from $\xxx$ to $\zzz$ with boundary $\brho$, so that $\yyy$ represents $I_{\alpha}(\xxx)a(\brho)I_{\alpha}(\zzz)$, where $I_{\alpha}$ is the idempotent corresponding to $o_{\alpha}$ for any generator. Then $I_{\alpha}(\xxx)\cdot \yyy = \yyy$ counts half-strips from $I_{\alpha}(\xxx)$ to $\yyy$ with the same boundary $\brho$. Since $\xxx$ and $I_{\alpha}(\xxx)$ occupy the same $\alpha$-arcs, then  $J(o_\alpha(\xxx))  = J(o_\alpha(I_{\alpha}(\xxx)))$. Similarly, $J(o_\alpha(\yyy))  = J(o_\alpha(\zzz))$, $J(o_\beta(\xxx)) = J(o_\beta(\zzz))$, and $J(o_\beta(I_{\alpha}(\xxx))) = J(o_\beta(\yyy))$. 
Let
\begin{align*}
J(o_\alpha(\xxx)) &= (a_1, \ldots, a_p)\\
J(o_\alpha(\zzz)) &= (b_1, \ldots, b_p)\\
J(o_\beta(\xxx)) &= (c_1, \ldots, c_p)\\
J(o_\beta(\yyy)) &= (d_1, \ldots, d_p),\\
\end{align*}
and write $I_{\alpha}(\xxx) = (w_1, \ldots, w_p)$,  again with $w_i$ ordered according to the occupied $\alpha$-arcs.

Let $\sigma_{\alpha}$ be the permutation that maps $J(o_\alpha(\zzz))$ to $J(o_\alpha(\xxx))$ along half-strips, i.e. $\sigma_{\alpha}(i) = i$ if $x_i = z_i$, and $\sigma_{\alpha}(i) = j$ if $x_j$ and $z_i$ are connected by a half-strip. Let $\sigma_{\beta}$ be the permutation that maps $J(o_\beta(\xxx))$ to $J(o_\beta(I_{\alpha}(\xxx)))$ so that $\sigma_{\beta}(i) = j$ if $x_i$ and $w_j$ occupy the same $\alpha$-arc.
Then $\sigma_{\zzz} = \sigma_{\xxx}\circ \sigma_{\alpha}$, $\sigma_{\yyy} = \sigma_{\beta}\circ \sigma_{\xxx}\circ \sigma_{\alpha}$, and $\sigma_{I_{\alpha}(\xxx)} = \sigma_{\beta}\circ \sigma_{\xxx}$, so  $\sigma_{\xxx}^{-1}\sigma_{\zzz} \sigma_{\yyy}^{-1}\sigma_{I_{\alpha}(\xxx)} = \id$. Hence,   $\textrm{sign}(\sigma_{\xxx})\textrm{sign}(\sigma_{\zzz}) \textrm{sign} (\sigma_{\yyy})\textrm{sign} (\sigma_{I_{\alpha}(\xxx)} )= 1$.

In addition, if $x_i = z_{\sigma_{\alpha}(i)}$, then $w_i = y_{\sigma_{\alpha}(i)}$ and it follows that  $s(x_i) = s(z_{\sigma_{\alpha}(i)})$ and  $s(w_i) = s(y_{\sigma_{\alpha}(i)})$. If instead $x_i$ and $z_{\sigma_{\alpha}(i)}$ are connected by a half-strip, then $w_i$ and $y_{\sigma_{\alpha}(i)}$ are connected by a half-strip, with the same boundary, so $w_i, x_i, z_{\sigma_{\alpha}(i)}, y_{\sigma_{\alpha}(i)}$ are the vertices of a rectangle, so $s(w_i) s(x_i) s(z_{\sigma_{\alpha}(i)})s(y_{\sigma_{\alpha}(i)}) = 1$.

Thus
\begin{align*}
s(\xxx)s(\yyy)s(\zzz)s(I_{\alpha}(\xxx)) &=  \left(\textrm{sign}(\sigma_{\xxx})\prod_{i=1}^p s(x_i)\right) \cdot \left( \textrm{sign}(\sigma_{\yyy})\prod_{i=1}^p s(y_i)\right) \\
&\hspace{20pt} \cdot\left( \textrm{sign}(\sigma_{\zzz})\prod_{i=1}^p s(z_i)\right)\cdot \left(  \textrm{sign}(\sigma_{I_{\alpha}(\xxx)})\prod_{i=1}^p s(w_i) \right)\\
& = \textrm{sign}(\sigma_{\xxx})\textrm{sign}(\sigma_{\zzz}) \textrm{sign} (\sigma_{\yyy})\textrm{sign} (\sigma_{I_{\alpha}(\xxx)} )\\
& \hspace{20pt}  \cdot \prod_{i=1}^p  \left(s(w_i) s(x_i) s(z_{\sigma_{\alpha}(i)})s(y_{\sigma_{\alpha}(i)}) \right)\\
& = 1.
\end{align*}
On the other hand, $I_{\alpha}(\xxx)$ is an idempotent, so $\sigma_{I_{\alpha}(\xxx)} = \id$, and $s(w_i) = 1$ for all $i$, and so $s(I_{\alpha}(\xxx))  = 1$. Therefore, 
\begin{equation*}
s(\xxx)s(\yyy) = s(\zzz).\qedhere
\end{equation*}
\end{proof}

Next, we define a $\Z/2$ grading on $\cfdhat$. Given a Heegaard diagram $\HH$ for a bordered $3$-manifold, order  the $\alpha$ arcs as above,  but according to the orientation on $-\bdy \HH$, and orient them from $\alpha_i^+$ to $\alpha_i^-$. Also order and orient all $\alpha$ and $\beta$ circles, and define a complete ordering on all $\alpha$-curves by $\alpha_1, \ldots, \alpha_{2k}, \alpha^c_1, \ldots, \alpha^c_{g-k}$.

Write generators as ordered tuples $\xxx = (x_1, \ldots, x_g)$ to agree with the ordering of the occupied $\alpha$-curves, and for any generator $\xxx$,  define $\sigma_{\xxx}$ to be the permutation for which 
\begin{align*}
x_1 &\in \alpha_{i_1}\cap \beta_{\sigma_{\xxx}(1)}\\
& \hspace{7pt} \vdots\\
x_k &\in \alpha_{i_k}\cap \beta_{\sigma_{\xxx}(k)}\\
x_{k+1} &\in \alpha_1^c\cap \beta_{\sigma_{\xxx}(k+1)}\\
&\hspace{7pt}  \vdots\\
 x_g &\in \alpha_{g-k}^c\cap \beta_{\sigma_{\xxx}(g)},
\end{align*}
where $(i_1, \ldots, i_k) = J(o(\xxx))$.
For any $x_i$, define $s(x_i)$ to be the intersection sign of $\balpha$ and $\bbeta$ at $x_i$.  We also define  $\sigma_\sss$ for each $k$-element set $\sss\subset [2k]$ to be the permutation in $S_{2k}$ that maps the ordered set $(1,\ldots, k)$ to $J(\sss)$ and $(k+1, \ldots, 2k)$ to $J([2k]\setminus \sss)$.

Last, define the sign of $\xxx$ by 
$$s(\xxx) = \textrm{sign}(\sigma_{o(\xxx)})\textrm{sign}(\sigma_{\xxx})\prod_{i=1}^g s(x_i).$$

\begin{proposition}\label{scfd}
The sign function $s$
induces a differential grading $m$ on $\cfdhat$, in the sense that the unique function $m: \mathfrak S(\HH)\to \Z/2$ for which $s= (-1)^m$ is a differential grading.
\end{proposition}

In other words, we can define the $\Z/2$ grading of a generator to be $0$ if its sign is $1$, and $1$ if its sign is $-1$.

To prove this, we glue $\textrm{AZ}(\zz)$  to $\HH$ and define signs for the generators of the resulting Heegaard diagram. 
Define a total ordering on the $\alpha$-curves to agree with the ordering on $\HH$ and on the  $\beta$-curves by concatenating the ordering on $\textrm{AZ}(\zz)$ and the ordering on $\HH$ (in this order). The orientations on the $\alpha$-arcs in both diagrams are compatible, and induce an orientation on the $\alpha$-circles obtained after gluing.
Then we can define permutations and local intersection signs as above, and  define a sign for each generator $\zzz = \xxx \otimes \yyy$ of 
$\mathrm{AZ}(\zz)\cup \HH$ by 
$$s(\zzz) := \textrm{sign}(\sigma_{\zzz})\prod_{i=1}^{g+k}s(z_i).$$ 
Note that this definition induces a relative $\Z/2$ Maslov grading:

\begin{lemma}\label{sazh}
If two generators $\zzz$ and $\zzz'$ in $\mathrm{AZ}(\zz)\cup \HH$ are connected by an index $1$ domain that doesn't touch the boundary of the Heegaard diagram, then they have opposite signs. 
\end{lemma}
\begin{proof}
Glue a diagram $\HH_{\beta}$ to $\mathrm{AZ}(\zz)\cup \HH$ to obtain a closed diagram, and pick a generator ${\bf{b}}$ on $\HH_{\beta}$ so that $\bbb\otimes \zzz$ and $\bbb\otimes \zzz'$ are generators in the closed diagram. These new generators are connected by the same domain, so their  Maslov gradings differ by one. For any choice of completing the ordering and orientation on the $\alpha$ and $\beta$ curves for the closed diagram, define signs for the generators  as above, and observe that our sign definition agrees with that of  \cite[Section 2.4]{dsfh} (we can think of a closed Heegaard diagram for a manifold $Y$ as a sutured diagram for $Y(1)$ by removing an open neighborhood of the basepoint), so 
 any way of completing the ordering and orientation will yield $s(\bbb\otimes \zzz) = - s(\bbb\otimes \zzz')$. On the other hand,

 $$\textrm{sign}(\sigma_{\zzz}) \textrm{sign}(\sigma_{\zzz'}) =  \textrm{sign}(\sigma_{\bbb\otimes \zzz}) \textrm{sign}(\sigma_{\bbb\otimes\zzz'})$$
and 
$$\prod_{i=1}^{g+2k}(\bbb\otimes \zzz)_i\cdot \prod_{i=1}^{g+2k}(\bbb\otimes \zzz)_i = \prod_{i=1}^k\bbb_i \cdot \prod_{i=1}^{g+k}\zzz_i\cdot \prod_{i=1}^k\bbb_i \cdot \prod_{i=1}^{g+k}\zzz'_i  = \prod_{i=1}^{g+k}\zzz_i\cdot \prod_{i=1}^{g+k}\zzz'_i,$$
so 
\begin{equation*}
s(\zzz)s(\zzz') = s(\bbb\otimes \zzz)s(\bbb\otimes \zzz') = -1.  \qedhere
\end{equation*}
\end{proof}

\begin{lemma}\label{sadds}
With the above definitions,  if $\xxx\otimes \yyy$ is a generator of $\mathrm{AZ}(\zz)\cup \HH$, then 
$$s(\xxx\otimes \yyy) = s(\xxx)s(\yyy).$$
\end{lemma}

\begin{proof}
Let $\bar\sigma_{\xxx}\in S_{g+k}$ be the concatenation of $\sigma_{\xxx}$ and $\id_{S_g}$, and let $\bar\sigma_{\yyy}\in S_{g+k}$ be the concatenation of $\id_{S_k}$  and $\sigma_{\yyy}$. Note that $\bar\sigma_{\xxx}\circ \bar\sigma_{\yyy}$ is the concatenation of $\sigma_{\xxx}$ with $\sigma_{\yyy}$, and also equals $\bar\sigma_{\yyy}\circ \bar\sigma_{\xxx}$.
Denote $\xxx\otimes \yyy$  by $\zzz$ and observe that  $\sigma_{\zzz} = \bar\sigma_{\xxx}\circ \bar\sigma_{\yyy}\circ \sigma_{o(\yyy)}$, and so 
\begin{align*}
\textrm{sign} (\sigma_{\zzz} ) &=     \textrm{sign}(\bar\sigma_{\xxx}) \textrm{sign} (\bar\sigma_{\yyy})\textrm{sign}( \sigma_{o(\yyy)})  \\
& =  \textrm{sign}(\sigma_{\xxx}) \textrm{sign} (\sigma_{\yyy})\textrm{sign} (\sigma_{o(\yyy)})
\end{align*}
Then
\begin{align*}
s(\xxx)s(\yyy) &=   \textrm{sign}(\sigma_{\xxx})\prod_{i=1}^k s(x_i) \cdot  \textrm{sign} (\sigma_{o(\yyy)})  \textrm{sign}(\sigma_{\yyy})\prod_{i=1}^g s(y_i)\\
& = \textrm{sign}(\sigma_{\xxx}) \textrm{sign}(\sigma_{\yyy}) \textrm{sign} (\sigma_{o(\yyy)})  \prod_{i=1}^{g+k}z_i\\
& =  \textrm{sign} (\sigma_{\zzz} )   \prod_{i=1}^{g+k}z_i   \qedhere
\end{align*}

\end{proof}

\begin{proof}[Proof of Proposition \ref{scfd}]
Suppose $\aaa\yyy$ is in the differential of $\xxx$, so there are a $B\in \pi_2(\xxx,\yyy)$ and a sequence of Reeb chords $\rho$, such that $(B, \rho)$ is compatible,  $\textrm{ind} (B, \rho)  =1$, and $\aaa = a(\rho)$. Then if we glue $\mathrm{AZ}(\zz)$ to $\HH$, $B$ completes to a closed index $1$ domain from $I_D(\xxx)\otimes \xxx$ to $\aaa\otimes \yyy$, where we think of $I_D(\xxx)$ and $\aaa$ as generators of $\mathrm{AZ}(\zz)$. 

Lemmas \ref{sazh} and \ref{sadds}, along with the fact that idempotents have grading $0$, immediately imply that 
\begin{equation*}
-1 = s(I_D(\xxx)\otimes \xxx)s(\aaa\otimes \yyy) =  s(I_D(\xxx))s(\xxx)s(\aaa)(\yyy) = s(\xxx)s(\aaa)(\yyy). \qedhere
\end{equation*}
\end{proof}

\remark To conclude this section, we explain how to relate the $\Z/2$ grading from this section to the grading from Section \ref{z2sec}.
Recall $I(\sss)$ is in $I(\zz,0)$ whenever $|\sss|=k$. Given such $\sss$, look at $J(\sss) = (s_1, \ldots, s_k)$,  let $\rho_i^{\sss}$ be the Reeb chord from $\alpha_i^-$ to $\alpha_{s_i}^-$ whenever $i\neq s_i$,  and let $\brho^{\sss}$ be the set of all such Reeb chords. Choose grading refinement data $\sss_0:=\{1, \ldots, k\}$  and define $\psi(\sss):=\gr'(a(\brho^{\sss})) = (\iota(a(\brho^{\sss})); [\brho^{\sss}])$ for every other $\sss$. This specifies a refined grading $\gr$ on $\cala(\zz, 0)$. The resulting $\Z/2$ grading $m = f\circ \gr$ obtained by composing with the map $f$ from Section \ref{z2sec} agrees with $s$.  In other words, given $\aaa\in \cala(\zz, 0)$, then $s(\aaa) = (-1)^{m(\aaa)}$, and, for an appropriate choice of a base generator for $\cfdhat(\HH)$ in each $\mathrm{spin}^c$ structure, $ s(\xxx)= (-1)^{m(\xxx)}$. The proof that the two gradings agree is a rather tedious computation of the refined gradings of the half-strip domains on $\mathrm{AZ}(\zz)$, and, since it does not affect the results of this paper, we do not include it.

 \section{The Euler characteristic of bordered Heegaard Floer homology}\label{chiker}

In this section, we prove Theorem \ref{cfdker}.

 Let $(Y, \zz, \phi)$ be a bordered $3$-manifold with Heegaard diagram $\mathcal H$ (so $\bdy \mathcal H =- \zz$) of genus $g$, and let $k$ be the genus of $F(\zz)\cong \bdy Y$.   For simplicity, we assume that $\zz$ is the split pointed matched circle. At the end of this section we provide a simple handle slide argument to complete the proof of Theorem \ref{cfdker} for general $\zz$.

 Fix an ordering and orientation of all $\beta$-circles, $\alpha$-circles, and $\alpha$-arcs. Let  $M(\mathcal H)$ be the $(g+k)\times g$ signed intersection matrix given by 
$$ m_{ij} = \left\{ \begin{array}{ll} 
\#(\alpha_i\cap \beta_j) & \textrm{ if }  i\leq g-k\\
\#(\alpha_i^a\cap \beta_j) & \textrm{ if }  i> g-k.
\end{array} \right.
$$
We can read $[\cfdhat(Y)]$ from  $M(\mathcal H)$ in the following way. Fix a $k$-element subset $\sss\subset [2k]$ and let $M(\mathcal H)_{\sss}$ be the square matrix obtained from $M(\mathcal H)$ by deleting the rows corresponding to $\sss$. More precisely, if $i\in \sss$, we delete the $(g-k+i)^{\textrm{th}}$ row, i.e. the one corresponding to $\alpha^a_i$. Observe that  

$$ \det M(\mathcal H)_{\sss} =  \textrm{sign}(\sigma_{o(\xxx)}) \sum_{y\in\mathfrak S(\mathcal H), I_D(y) = I(\sss)}s(y) =  \textrm{sign}(\sigma_{o(\xxx)})  \sum_{y\in\mathfrak S(\mathcal H), I_D(y) = I(\sss)}(-1)^{m(y)}.$$

In other words,  $\det M(\mathcal H)_{\sss}$ is $\pm$ the coefficient of $a_\sss$ in  $[\cfdhat(Y)]$, where $a_\sss$ is the basis generator for $\Lambda^\ast H_1(F; \Z)$ defined in Section \ref{k0sec} corresponding to the set $\sss$.

Next, we relate $\ker (H_1(F(\zz))\to H_1(Y))$ to $M(\mathcal H)$. If we cap off $\mathcal H$ with a disk,  and close off each $\alpha^a_i$ inside the disk, we get a closed surface $\Sigma'$ of genus $g$ with $\{\alpha_1,\ldots, \alpha_{g-k}, \alpha'_1, \ldots, \alpha'_{2k}\}$ spanning a $(g+k)$-dimensional subspace of $H_1(\Sigma')$.  By adding in circles $\gamma_1,\ldots, \gamma_{g-k}$ such that 
\begin{align*}
\#(\gamma_i\cap \alpha_j) &= \left\{ \begin{array}{ll} 
1 & \textrm{ if }  i=j\\
0 & \textrm{ otherwise,}
\end{array} \right. \\
\#(\gamma_i\cap \alpha'_j) &= 0,
\end{align*}
we extend to a basis $\{\alpha_1,\ldots, \alpha_{g-k}, \gamma_1, \ldots, \gamma_{g-k},  \alpha'_2, \alpha'_1, \ldots, \alpha'_{2k}, \alpha'_{2k-1}\}$, which is dual to $\{\gamma_1, \ldots, \gamma_{g-k}, \alpha_1,\ldots, \alpha_{g-k}, \alpha'_1, \ldots, \alpha'_{2k}\}$. In other words, if 
$$\beta_i = \sum_{j=1}^{g-k} c_{ij}\gamma_j + \sum_{j=1}^{g-k} a_{ij}\alpha_j + \sum_{j=1}^{2k} a'_{ij}\alpha'_j,$$
then $\#(\beta_i\cap \alpha_j) = c_{ij}$, etc. Permute the last $2k$ rows of $M(\mathcal H)$ by swapping adjacent rows in pairs, i.e. for each $i$ exchange the rows corresponding to $\alpha'_{2i-1}$ and $\alpha'_{2i}$.  Call the new matrix $M'(\mathcal H)$. 
Let $\balpha$, $\boldsymbol{\alpha'}$, $\bbeta$, and  $\bgamma$ be the sets of $\alpha$, $\alpha'$, $\beta$, and $\gamma$ circles, respectively.   
The columns of $M'(\mathcal H)$ now represent $\beta$ circles as linear combinations of $\gamma$ and $\alpha'$ circles in the space $H_1(\Sigma')/H_1(\balpha)$. Note that the inclusion $H_1(\bdy Y)\to H_1(Y)$ is the composition of 
$$H_1(F)\hookrightarrow H_1(\Sigma')\twoheadrightarrow H_1(Y),$$
where the first map is the inclusion of the subspace $\Z\left<\boldsymbol{\alpha'}\right>$, and the second map is the quotient by all $\alpha$ and $\beta$ circles. In terms of our basis,
$$ \Z\left<\boldsymbol{\alpha'}\right>\overset{\iota}{\hookrightarrow} \Z\left<\bgamma, \balpha, \boldsymbol{\alpha'}\right>\overset{q}{\twoheadrightarrow} \Z\left<\bgamma, \balpha, \boldsymbol{\alpha'}\right>/\Z\left<\balpha, \bbeta\right>,$$
where $q$ is the quotient by the homology subspace generated by $\balpha\cup\bbeta$. Then $\ker (H_1(\bdy Y)\to H_1(Y)) = \ker (q\circ \iota)$, or if we first quotient by the space generated by $\balpha$, then $\ker (H_1(\bdy Y)\to H_1(Y)) = \ker(\bar q\circ \bar \iota)$ in the resulting sequence
$$ \Z\left<\boldsymbol{\alpha'}\right>\overset{\bar{\iota}}{\hookrightarrow} \Z\left<\bgamma, \boldsymbol{\alpha'}\right>\overset{\bar q}{\twoheadrightarrow} \Z\left<\bgamma, \boldsymbol{\alpha'}\right>/\Z\left<\bbeta\right>.$$
Define $V_{\beta} := \mathrm{span}\{\beta_1, \ldots, \beta_{g-k}\}\subset  \Z\left<\bgamma, \boldsymbol{\alpha'}\right>$ and $V_{\gamma} := \mathrm{span}\{\gamma_1, \ldots, \gamma_{g-k}\}\subset \Z\left<\bgamma, \boldsymbol{\alpha'}\right>$.
Now, $\ker (\bar q\circ \bar \iota)$ is isomorphic under $\bar\iota$ to $\im \bar \iota \cap \ker \bar q$, i.e. to  the subspace of $V_{\beta}$ that is perpendicular to $V_{\gamma}$. In other words,
$$\ker (\bar q\circ \bar \iota) = \{\pi_{\alpha'}(v)|v\in V_{\beta},  \pi_{V_{\gamma}}(v)=0\}.$$
We can change basis for $V_{\beta}$ by performing column operations on $M'(\mathcal H)$ so that     $\im \bar \iota \cap \ker \bar q$       is generated by the initial columns. This corresponds to handleslides of $\beta$-circles over $\beta$-circles, so the Heegaard diagram after the handleslides specifies the same bordered manifold. Thus, we may assume that $M'(\mathcal H)$ already has this form, i.e.  that $\im \bar \iota \cap \ker \bar q$       is generated by the initial columns of $M'(\mathcal H)$.
\begin{lemma}\label{rank}
Let $M^{top}$ be the submatrix of $M'(\mathcal H)$ formed by the top $g-k$ rows. The rank of $M^{top}$ is $g-k$ iff $H_1(Y, \bdy Y)$ is finite. 
\end{lemma}

\begin{proof}
Pick $\delta_1, \ldots, \delta_g$ dual to $\beta_1, \ldots, \beta_g$. The rows of $M^{top}(\mathcal H)$ record the intersections of a given  $\alpha$-circle with the $\beta$-circles, so they represent the linear combination of that $\alpha$-circle in terms of the $\delta$-circles.  
 
By the universal coefficients theorem, $H_1(Y, \bdy Y)$ is finite if and only if $H^1(Y, \bdy Y) = 0$, and by Poincar\'e-Lefschetz duality, $H^1(Y, \bdy Y)\cong H_2(Y)$. Let $\balpha$, $\bbeta$, and $\bdelta$ be the sets of $\alpha$-circles, $\beta$-circles, and $\delta$-circles, respectively.    The Mayer-Vietoris sequence for the Heegaard decomposition of $Y$ specified by $\mathcal H$ identifies $H_2(Y)$ with $\ker (H_1(\balpha)\oplus H_1(\bbeta)\to H_1(\Sigma))$, so finally, $H_1(Y, \bdy Y)$ is finite if and only if $\ker (H_1(\balpha)\oplus H_1(\bbeta)\to H_1(\Sigma))= 0$. But $H_1(\Sigma) = H_1(\bdelta)\oplus H_1(\bbeta)$, so the kernel is zero-dimensional exactly when $\pi_{\delta}H_1(\balpha)$ has dimension $g-k$. The projection $\pi_{\delta}H_1(\balpha)$ is exactly the span of the rows of $M^{top}(\mathcal H)$, so the dimension of the projection is the rank of $M^{top}(\mathcal H)$. 
\end{proof} 

Similarly, let $M^{bottom}$ be the submatrix of $M'(\mathcal H)$ formed by the bottom $2k$ rows. 
Note that for any $\sss$, $\rank M'(\mathcal H)_{\sss}\leq \rank M^{top} + \rank M^{bottom}$, so if $\rank H_1(Y, \bdy Y)>0$, then $\det(M'(\mathcal H)_{\sss})=0$, and so $[\cfdhat(Y)] = 0$. If $\rank H_1(Y, \bdy Y)=0$, then by Lemma \ref{rank} the matrix $M'(\mathcal H)$ has the block form 
\begin{displaymath}
\left(\begin{array}{c|c}
0 & B \\
\hline
A & C
\end{array}\right),
\end{displaymath}
where $0$ is the $(g-k)\times k$ zero matrix, and $B$ is a $(g-k)\times(g-k)$ matrix with $\det(B) = |H_1(Y, \bdy Y)|$. To see that $\det(B) = |H_1(Y, \bdy Y)|$, observe that the columns of $B$ represent  $\beta$-circles as linear combinations of $\gamma$ circles, after quotienting by all $\alpha'$ circles. On the other hand, by definition we have
\begin{align*}
H_1(Y,\bdy Y) &\cong H_1(Y)/\im(H_1(\bdy Y)\to H_1(Y))\\
&\cong ( \Z\left<\bgamma, \boldsymbol{\alpha'}\right>/\Z\left<\bbeta\right>)/\im  (\bar q\circ \bar \iota)\\
&\cong \Z\left<\bgamma,  \boldsymbol{\alpha'}\right>/\Z\left<\bbeta,  \boldsymbol{\alpha'}\right>\\
&\cong \Z\left<\bgamma\right>/\Z\left<\bbeta\right>\\
&\cong \Z^{g-k}/\im B.
\end{align*}
Since $H_1(Y,\bdy Y)$ is finite, then $B$ has full rank, and $|H_1(Y,\bdy Y)| = \det B$.

We already discussed that the columns of  
$\left(\begin{array}{c}
0 \\
\hline
A
\end{array}\right)
$
span $\im \bar \iota \cap \ker \bar q$ and the columns $c_1,\ldots, c_k\in \R^{2k}$ of $A$ span $\ker (\bar q\circ \bar \iota)$. Last, 
\begin{align*}
\Lambda^k \ker (\bar q\circ \bar \iota) & = \{v_1\wedge\cdots \wedge v_k| v_i\in \ker (\bar q\circ \bar \iota)  \}\\
&= \{\sum_{i=1}^k t_{1i} c_i\wedge\cdots\wedge \sum_{i=1}^k t_{ki} c_i | t_{ij}\in \Z \}\\
&= \{\sum_{\sigma\in S_k} t_{1\sigma(1)} c_{\sigma(1)}\wedge\cdots\wedge t_{k\sigma(k)} c_{\sigma(k)} | t_{ij}\in \Z \}\\
&= \mathrm{span} \{\sum_{\sigma\in S_k}  c_{\sigma(1)}\wedge\cdots\wedge  c_{\sigma(k)}\}\\
&= \mathrm{span}\{c_1\wedge\cdots\wedge  c_k\}\\
&= \mathrm{span}\{\sum_{i=1}^k a_{i1} e_i\wedge\cdots\wedge \sum_{i=1}^k a_{ik} e_i\}\\
&= \mathrm{span}\{\sum_{\sss\subset [2k]}\sum_{\sigma\in S_k}^k a_{\sigma(j)j}\textrm{sgn} (\sigma) e_{i_1}\wedge\cdots\wedge  e_{i_k}\}\\
&= \mathrm{span} \{\sum_{\sss\subset [2k]}\det A_{\sss} \cdot e_{i_1}\wedge\cdots\wedge  e_{i_k}\}\\
&= \mathrm{span} \{\sum_{\sss\subset [2k]} \det A_{\sss} a_\sss\}.
\end{align*}
 But $\det M(\mathcal H)_{\sss} = \det A_{\sss}\det B$, so 
\begin{align*}
\mathrm{span}[\cfdhat(Y, \zz, \phi)] &= \det B\cdot \Lambda^k \ker(H_1(F(\zz))\to H_1(Y))\\
&=  |H_1(Y, \bdy Y)| \Lambda^k \ker(H_1(F(\zz))\to H_1(Y)).
\end{align*}

This proves Theorem \ref{cfdker} in the case when $\zz$ is a split circle. The promised handle slide argument is merely the observation that arc slides correspond to row operations, and any pointed matched circle for a surface of genus $k$ can be obtained from the split one by a sequence of arc slides. Specifically, sliding $\alpha'_i$ over $\alpha'_j$ corresponds to adding the $(g-k+j)^{\textrm{th}}$ row to the $(g-k+i)^{\textrm{th}}$ row in $M(\mathcal H)$, and to adding the $(g-k+i)^{\textrm{th}}$ row to the $(g-k+j)^{\textrm{th}}$ row in $M'(\mathcal H)$. Row operations preserve determinants, and this completes the proof of Theorem \ref{cfdker}.

\bibliographystyle{/Users/inapetkova/Documents/work/hamsplain2}

\bibliography{/Users/inapetkova/Documents/work/master}

\end{document}